%% file: manuscript.tex
\documentclass[11pt,a4paper]{article}
\usepackage[backend=bibtex, citestyle=authoryear]{biblatex}  

\input{setting.tex}

\input{tikz.tex}
\definecolor{red}{rgb}{0.6, 0.0, 0.0}
\title{A piecewise deterministic Monte Carlo method for diffusion bridges}
\author[1]{Joris Bierkens\orcidA{}}
\author[1]{Sebastiano Grazzi\thanks{Corresponding author. E-mail address: \href{mailto:s.grazzi@tudelft.nl}{\texttt{s.grazzi@tudelft.nl}}.}\orcidB{}}
\author[1]{\\Frank van der Meulen\orcidC{}}
\author[2]{Moritz Schauer\orcidD{}}

\affil[1]{Delft Institute of Applied Mathematics (DIAM), Delft University of Technology}
\affil[2]{Department of Mathematical Sciences, Chalmers University of Technology and University of Gothenburg}

\date{\today}
\begin{document}
\maketitle
\input{part1.tex} 
\input{part2.tex} 
\input{part3.tex} 
\input{part4.tex} 
\input{part5.tex} 
\input{part6.tex} 
\input{part7.tex} 
\input{part8.tex} 
\appendix
\input{appendix.tex}
  \printbibliography 
\end{document}

%% file: setting.tex
\usepackage[utf8]{inputenc}

\usepackage{graphicx}

\usepackage[a4paper, total={6in, 9in}]{geometry}

\usepackage{algorithm}
\usepackage{algpseudocode}

\usepackage{enumitem}

\usepackage[dvipsnames]{xcolor}

\usepackage{bm}
\usepackage{bbm}
\usepackage{amssymb,amsmath,amsfonts,amsthm,epsfig,epstopdf,url,array}

\usepackage[toc,page]{appendix}
\usepackage{authblk}

\usepackage{float}
\usepackage{tikz}
\usepackage{tkz-euclide}
\usepackage{pgfplots}
\pgfplotsset{compat=newest}

    \usepackage[colorlinks=true,linkcolor=red,citecolor=blue,urlcolor=black]{hyperref}

\definecolor{lime}{HTML}{A6CE39}
\DeclareRobustCommand{\orcidicon}{%
	\begin{tikzpicture}
	\draw[lime, fill=lime] (0,0) 
	circle [radius=0.16] 
	node[white] {{\fontfamily{qag}\selectfont \tiny ID}};
	\draw[white, fill=white] (-0.0625,0.095) 
	circle [radius=0.007];
	\end{tikzpicture}
	\hspace{-2mm}
}

\foreach \x in {A, ..., Z}{%
	\expandafter\xdef\csname orcid\x\endcsname{\noexpand\href{https://orcid.org/\csname orcidauthor\x\endcsname}{\noexpand\orcidicon}}
}

\DeclareMathOperator*{\argmin}{arg\,min}
\DeclareMathOperator\supp{supp}
\theoremstyle{plain}
\newtheorem{thm}{Theorem}[section]
\newtheorem{lem}[thm]{Lemma}
\newtheorem{prop}[thm]{Proposition}

\theoremstyle{definition}
\newtheorem{assumption}[thm]{Assumption}
\newtheorem{defn}[thm]{Definition}

\newtheorem{exmp}[thm]{Example}
\theoremstyle{remark}
\newtheorem{rmk}[thm]{Remark}

\newcommand{\dd}{ \mathrm{d}}
\newcommand{\EE}{\mathbb{E}}

\newcommand{\PP}{\mathbb{P}}
\newcommand{\ind}{\mathbf{1}}
\newcommand{\RR}{\mathbb{R}}

\newcommand{\QQ}{\mathbb{Q}}
\newcommand\independent{\protect\mathpalette{\protect\independenT}{\perp}}
\def\independenT#1#2{\mathrel{\rlap{$#1#2$}\mkern2mu{#1#2}}}


%% file: tikz.tex
\usetikzlibrary{backgrounds}
\usetikzlibrary{arrows,decorations.markings}
\usetikzlibrary{arrows.meta}
\usetikzlibrary{shapes,shapes.geometric,shapes.misc}
\usepackage{pst-eucl}

\tikzstyle{tikzfig}=[baseline=-0.25em,scale=0.5]

\pgfkeys{/tikz/tikzit fill/.initial=0}
\pgfkeys{/tikz/tikzit draw/.initial=0}
\pgfkeys{/tikz/tikzit shape/.initial=0}
\pgfkeys{/tikz/tikzit category/.initial=0}

\pgfdeclarelayer{edgelayer}
\pgfdeclarelayer{nodelayer}
\pgfsetlayers{background,edgelayer,nodelayer,main}

\tikzstyle{none}=[inner sep=0mm]

\newcommand{\tikzfig}[1]{%
{\tikzstyle{every picture}=[tikzfig]
\IfFileExists{#1.tikz}
  {\input{#1.tikz}}
  {%
    \IfFileExists{./figures/#1.tikz}
      {\input{./figures/#1.tikz}}
      {\tikz[baseline=-0.5em]{\node[draw=red,font=\color{red},fill=red!10!white] {\textit{#1}};}}%
  }}%
}

\tikzstyle{every loop}=[]

%% file: part1.tex
\subsection*{Abstract}
We introduce the use of the Zig-Zag sampler to the problem of sampling conditional diffusion processes (diffusion bridges). The Zig-Zag sampler is a rejection-free sampling scheme based on a non-reversible continuous piecewise deterministic Markov process.
Similar to the L\'evy-Ciesielski construction of a Brownian motion, we expand the diffusion path in a truncated Faber-Schauder basis. The coefficients within the basis are sampled using a Zig-Zag sampler. A key innovation is the use of   the \textit{fully local} algorithm for the Zig-Zag sampler that allows to exploit the sparsity structure implied by the dependency graph of the coefficients and by the \textit{subsampling} technique to reduce the complexity of the algorithm. We illustrate the performance of the proposed methods in a number of examples. 

\noindent \textit{\textbf{Keywords:} diffusion bridge, conditional diffusion, diffusion process, Faber-Schauder basis, intractable target density, local Zig-Zag sampler, piecewise deterministic Monte Carlo, high dimensional simulation}

%% file: part2.tex
\section{Introduction}
Diffusion processes are an important  class of continuous time probability models which find applications in many fields such as finance, physics and engineering. They naturally arise by adding Gaussian random perturbations (white noise) to deterministic systems. We consider diffusions   described by a one-dimensional stochastic differential equation of the form
\begin{equation}
\label{sde}
\dd X_t = b(X_t)\dd t + \dd W_t, \quad X_0 = u,    
\end{equation}
where $(W_t)_{t\ge0}$ is a driving scalar Wiener process defined in some probability space and $b$ is the \textit{drift} of the process.  The solution of equation (\ref{sde}), assuming it exists, is an instance of one-dimensional time-homogeneous diffusion. We aim to sample $X$ on $[0,T]$ conditional on $\{X_T=v\}$, also known as  a \textit{diffusion bridge}.

One driving motivation for studying this problem is estimation for discretely observed diffusions. Here, one assumes observations $\mathcal{D}=\{x_{t_1},\ldots, x_{t_N}\}$ at observations times $t_1<\ldots < t_N$ are given and interest lies in estimation of a parameter $\theta$ appearing in the drift $b$. 
It is well known that this problem can be viewed as a missing data problem  as in \textcite{roberts2001inference}, where one iteratively imputes the missing paths conditional on the parameter and the observations, and then the parameter conditional on the ``full'' continuous path. Due to the Markov property, the missing paths in between subsequent observations can be sampled independently and each of such segments constitutes a diffusion bridge. 
As this application requires sampling iteratively many diffusion bridges, it is crucial to have a fast algorithm for this step. We achieve this by adapting the Zig-Zag  sampler for the simulation of diffusion bridges. The Zig-Zag  sampler is an innovative non-reversible and rejection-free Markov process Monte Carlo algorithm which can exploit the structure present in this high-dimensional sampling problem. It is based on simulating a piecewise deterministic Markov process (PDMP). To the best of our knowledge, this is the first application of PDMPs for diffusion bridge simulation. This method also illustrates the use of a local version of the Zig-Zag  sampler in a genuinely high dimensional setting (arguably even an infinite dimensional setting).

The problem of diffusion bridge simulation has received considerable attention over the past two decades, see for example \textcite{bladt2014simple}, \textcite{beskos2006retrospective}, \textcite{vandermeulen2017}, \textcite{mider2019simulating}, \textcite{bierkens2018simulation} and references therein. This far from exhaustive list of references includes methods that apply to a more general setting than considered here, such as multivariate diffusions, conditioning on partial observations and hypo-elliptic diffusions. 
Among the methods that can be applied, most of the methodologies available are of the acceptance-rejection type and scale  poorly with respect to some parameters of the diffusion bridge. For example, if the proposed path is not informed by the target distribution, the probability of accepting the path depends strongly on the discrepancy between the proposed path and the target diffusion bridge measure and usually scales poorly as the time horizon of the diffusion bridge $T$ grows. In contrast, gradient based techniques which compute informed proposals (e.g. Metropolis-adjusted Langevin algorithm), require the evaluation of the gradient of the target distribution, which, in this case, is a path integral that has to be generally computed numerically and its computational cost is of order $T$,  leading to computational limitations. The present work aims to alleviate such restrictions through the use of a rejection-free method and an exact subsampling technique which reduces the cost of  evaluating the gradient. On a more abstract level, our method can be viewed as targeting a probability distribution which is obtained by a push-forward of Wiener measure through a change of measure. It then becomes apparent that the studied problem of diffusion bridge simulation is a nicely formulated non-trivial  example  problem within this setting to study the potential of simulation based on PDMPs. Our results open new paths towards applications of the Zig-Zag for high dimensional problems.

\subsection{Approach} 
In this section we present the main ideas used in this paper.  
\subsubsection{Brownian motion expanded in the Faber-Schauder basis}
\label{sec:faber-schauder}
Our starting point is the L\'evy-Ciesielski construction of Brownian Motion. Define $\Bar{\phi}(t)=\sqrt{t}$, $\phi_{0,0}(t)=\sqrt{T}\left((t/T)\ind_{[0,T/2]}(t) +(1-t/T)\ind_{(1/2,1]}(t)\right)$  and set 
\[ \phi_{i,j}(t) = 2^{-i/2}\phi_{0,0}(2^i t - jT),\quad \text{ for }\quad  i = 0, 1,...,\quad  j= 0, 1,...2^{i}-1.
\] 
If $\Bar \xi$ is standard normal and $\{\xi_{i,j}\}$ is a sequence of independent standard normal random variables (independent of $\Bar \xi$), then
\begin{equation}
    \label{sum expansion}
    X^N(t) = \Bar{\phi}(t)\bar \xi+ \sum_{i=0}^N \sum_{j= 0}^{2^{i}-1}  \xi_{i,j} \phi_{i,j}(t)
\end{equation}
converges almost surely on $[0,T]$ (uniformly in $t$) to a Brownian motion as $N \to \infty$ (see e.g.\  Section 1.2 of \cite{mckean1969stochastic}). The basis formed by $\Bar{\phi}$ and $\{\phi_{i,j}\}$ is known as the Faber-Schauder basis (see Figure~\ref{bm}). The larger $i$, the smaller the support of $\phi_{i,j}$, reflecting that higher order coefficients represent the fine details of the process. A Brownian bridge starting in $u$ and ending in  $v$ can be obtained by fixing $\Bar \xi = v/\sqrt{T}$ and adding the function  $\Bar{\Bar{\phi}}(t)u = (1-t/T)u$ to \eqref{sum expansion}.  By sampling $\xi^N := (\xi_{0,0},\xi_{1,0},...,\xi_{N,2^N-1})$ (which in this case are standard normal),  approximate realisations of a Brownian bridge can be obtained.

\subsubsection{Zig-Zag sampler for diffusion bridges}
Let $\mathbb{Q}^u$ denote the Wiener measure on $C[0,T]$ with initial value $X_0 = u$ (cf.\ section 2.4 of \cite{karatzas1998brownian}) and let $\mathbb{P}^u$ denote the law on $C[0,T]$ of the diffusion in \eqref{sde}.  
 Under mild conditions on $b$, the two measures are absolutely continuous and their Radon-Nikodym derivative $\frac{\dd \mathbb{P}^u}{\dd \mathbb{Q}^u}$ is given by the Girsanov formula.
Denote by  $\mathbb{P}^{u, v_T}$ and $\mathbb{Q}^{u, v_T}$ the measures of the diffusion bridge and the Wiener bridge respectively, both starting at $u$ and conditioned to hit a point $v$ at time $T$. Applying the Bayes' law for conditional expectations (\cite{klebaner2005introduction}, Chapter~10) we obtain: 
\begin{equation}
    \label{bayes rule bridge}\frac{\dd\mathbb{P}^{u, v_T}}{\dd\mathbb{Q}^{u, v_T}}(X) = \frac{ q(0,u, T, v )}{ p(0,u, T, v )} 
     \frac{\dd\mathbb{P}^u}{\dd\mathbb{Q}^u} (X)  ,  
\end{equation}
where $p$ and $q$ are the transition densities of $X$  under $\mathbb{P}, \mathbb{Q}$ respectively so that for $s<t$, $p(s, x, t,  y)\dd y =  P(X_t \in \dd y \mid X_s = x)$. As $p$ is intractable, the Radon-Nikodym derivative for the diffusion bridge is only known up to proportionality constant.  The main idea now consists of rewriting the Radon-Nikodym derivative in \eqref{bayes rule bridge}, evaluating it in $X^N$  and running the Zig-Zag  sampler for $\xi^N$ targeting this density. Technicalities to actually get this to work are detailed in Section \ref{faber schauder expansion of diffusion processes}. A  novelty is the introduction of a {\it local} version of the   Zig-Zag  sampler, analogously to the \textit{local bouncy particle sampler} (\cite{2015arXiv151002451B}). This allows for exploiting the sparsity in the dependence structure of the coefficients of the Faber-Schauder expansion efficiently, resulting in a reduction of  the complexity of the algorithm. The methodology we propose is derived for  one dimensional diffusion processes with unit diffusivity. However, diffusions with state-dependent diffusivity can be transformed to this setting using the Lamperti transform (an example is given  in Subsection~\ref{Diffusion model with unbounded drift}). In Subsection~\ref{subsec: multivariate diffusion bridge} we generalize the method to multivariate diffusion processes  with unit diffusivity, assuming the drift to be  a conservative vector field.

\subsection{Contributions of the paper}
 The Faber-Schauder basis offers a number of attractive properties:
\begin{enumerate}[label=(\alph*)]
    \item The coefficients of a diffusions have a structural conditional independence property (see Section~\ref{A local Zig-Zag algorithm} and Appendix~\ref{App: fact diff bridge}) which can be exploited in numerical algorithms to improve their efficiency.
    \item A diffusion bridge is obtained from the unconditioned process by simply fixing the coefficient $\Bar \xi$.
    \item It will be shown (see for example Figure~\ref{figure qqplots}) that the non-linear component of the diffusion process is typically captured by  coefficients $\xi_{ij}$ in equation \eqref{sum expansion} for which $i$ is small. This allows for a low dimensional representation of the process and yet a good approximation. Therefore, the approximation error caused by  leaving out fine details is equally divided over $[0,T]$, contrary to approaches where a proxy for the diffusion bridge is simulated by Euler discretisation of an SDE governing its dynamics. In the latter case, the discretisation error accumulates over the interval on which the bridge is simulated.   
    \item It is very convenient from a computational point of view as each function is piecewise linear with compact support.
\end{enumerate}

We adopt the Zig-Zag  sampler (\cite{bierkens2019}) which is a sampler based on the theory of piecewise deterministic Markov processes (see \cite{fearnhead2018}, \cite{2015arXiv151002451B}, \cite{andrieu2019peskuntierney}, \cite{andrieu2018hypocoercivity}). The main reasons motivating this choice are:
\begin{enumerate}[label=(\alph*)]
    \item The partial derivatives of the log-likelihood of a diffusion bridge measure usually appear as a path integral that has to be computed numerically (introducing consequently computational burden derived by this step and its bias). The Zig-Zag sampler allows us to replace the gradient of the log-likelihood with an unbiased estimate of it without introducing bias in the target measure. This is done in Subsection \ref{Sampling diffusion bridges}  with the subsampling technique which was presented in \textcite{bierkens2019} for applications for which the evaluation of the log-likelihood is expensive due to the size of the dataset.
    \item In the same spirit as the local Bouncy Particle Sampler of \textcite{2015arXiv151002451B} and \textcite{Peters_2012}, the \textit{local} and the \textit{fully  local} Zig-Zag  sampler introduced in Section~\ref{A local Zig-Zag algorithm} reduces the complexity of the algorithm improving its efficiency with respect to the standard Zig-Zag Algorithm as the dimensionality of the target distribution increases (see Subsection~\ref{subsection: sca;omg for large T, N, d}). This opens the way to high dimensional applications of the Zig-Zag  sampler when the dependency graph of the target distribution is not fully connected and when using subsampling. The factorization of the log-likelihood and the local method we proposed is reminiscent of other work such as e.g. \textcite{Faulkner_2018}, \textcite{Michel_2019} and \textcite{Monmarch__2020}. 
    \item The method is a rejection-free sampler, differing from most of the methodologies available for simulating diffusion bridges.
    \item The Zig-Zag  sampler is defined and implemented in continuous time, eliminating the choice of tuning parameters appearing for example in the proposal density of the Metropolis-Hastings algorithm.
         This advantage comes at the cost of a more complicated method which relies upon bounding from above rates which are model specific and often difficult to derive (see Section \ref{Numerical Results} for our specific applications).
    \item The process is non-reversible: as shown, for example, in \textcite{diaconis2000analysis}, non-reversibility generally enhances the speed of convergence to the invariant measure and mixing properties of the sampler. For an advanced analysis on convergences results for this class of non-reversible processes, we refer to the articles \textcite{andrieu2019peskuntierney} and \textcite{andrieu2018hypocoercivity}.
\end{enumerate}

The local Zig-Zag  sampler relies on the conditional independence structure of the coefficients only. This translates to other settings than diffusion bridge sampling, or other choices of basis functions. For this reason, Section~\ref{A local Zig-Zag algorithm} describes the algorithms of the sampler in their full generality, without referring to our particular application. A documented implementation of the algorithms used in this manuscript can be found in \textcite{mschauer/ZigZagBoomerang.jl}.
 
\subsection{Outline} In Section~\ref{preliminaries} we set some notation and recap the Zig-Zag  sampler. In Section~\ref{faber schauder expansion of diffusion processes} we expand a diffusion process in the Faber-Schauder basis and prove the aforementioned conditional dependence.  The simulation of the coefficients $\xi^N$  presents some challenges as it is high dimensional and its density is expressed by an integral over the path. We give two variants of the Zig-Zag  algorithm which enables sampling in a high dimensional setting.
In particular, in Section~\ref{A local Zig-Zag algorithm} we present the local and fully  local Zig-Zag  algorithms which exploit a factorization of the joint density (Appendix~\ref{App: fact diff bridge}) and a subsampling technique  which, in this setting, is used  to  avoid the evaluation of the path integral appearing in the density (which otherwise would severely complicate the implementation of the sampler).  
In Section~\ref{Numerical Results} we illustrate our methodology using a variety of examples, validate our approach and compare the Zig-Zag  sampler with other benchmark MCMC algorithms. 
We conclude by sketching the extension of our method to multi-dimensional diffusion bridges, carrying out an informal scaling analysis and providing several remarks for future research (Section~\ref{Subsection: Informal scaling analysis} and Section~\ref{Conclusions}).

%% file: part3.tex
\section{Preliminaries}\label{preliminaries}
 Throughout, we denote by $\partial_i$ the partial derivative with respect to the coefficient $\xi_i$, the positive part of a function $f$ by $(f)^+$, the $i$th element and the Euclidean norm of a vector $x$ respectively by $[x]_i$ and $\| x\|$. The cardinality of a countable set $A$ is denoted by $|A|$.

\subsection{Notation for the Faber-Schauder basis}
To graphically illustrate the Faber-Schauder basis, a  construction of a Brownian motion with the representation of the basis functions is given in Figure~\ref{bm}. 
The Faber-Schauder  functions are piecewise linear with compact support. The length of the support and the height of the function is determined by the first index  while the second index determines the location. All basis functions with first index $i$ are referred to as \textit{level} $i$ basis functions.  For convenience, we often swap between double and single indexing of Faber-Schauder functions. Denote the double indexing with $(i,j)$ and the single indexing with $n$. We go from one to the other through the transformations
\[
i = \lfloor \log_2(n)\rfloor , \qquad j = n - 2^{i}, \qquad n = 2^{i} + j;
\]
where $\lfloor \cdot \rfloor$ denotes the floor function. The basis with truncation level $N$ has $M:=2^{N+1} - 1$ coefficients. Let $\xi^N$ denote the vector of  coefficients up to level $N$, i.e.\ 
\begin{equation}\label{eq:xiN} 
\xi^N := (\xi_{0,0},\xi_{1,0},...,\xi_{N,2^N-1}) \in \RR^{M}
\end{equation}
and let $X^{\xi^N} := X^N $ when we want to stress the dependencies of $X^N$ on the coefficients $\xi^N$. Using double indexing, we denote by  $S_{i,j} =  \supp \phi_{i,j}$.

\begin{figure}[ht]
    \centering
    \includegraphics[width=1\linewidth]{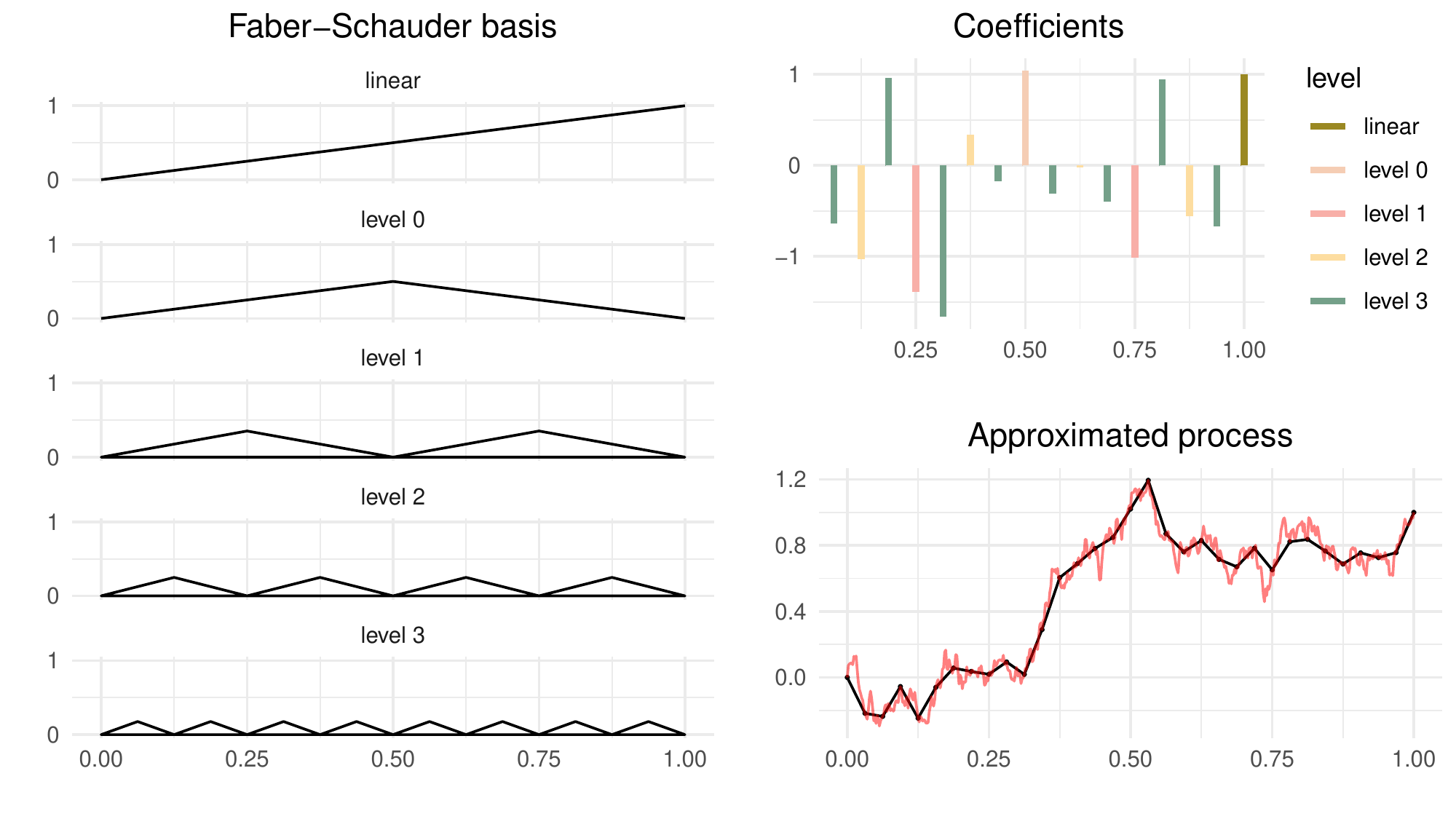}
    \caption{L\'evy-Ciesielski construction of a Brownian motion on $(0,1)$. On the left the Faber-Schauder  basis functions up to level $N =3$, on the top-right the values of the corresponding coefficients located at the peak of their relative FS basis function and on the bottom-right the resulting approximated Brownian path $X^N$ (black line) compared with a finer approximation (red line). The truncated sum defines the process in $2^{N+1} + 1$ finite dyadic points (black dots) with linear interpolation in between points. A finer approximation corresponds to Brownian fill-in noise between any two neighboring dyadic points.}
    \label{bm}
\end{figure}

\subsection{The Zig-Zag sampler}\label{zigzag sampler section}
A \textit{piecewise deterministic Markov process} (\cite{davis1993markov}) is a continuous-time process with behaviour  governed by random jumps at points in time, but deterministic evolution  governed by an ordinary differential equation in between those times (yielding piecewise-continuous realizations). If the  differential equation can be solved in closed form and the random event times can be sampled exactly, then the process can be  simulated in continuous time without introducing any discretization error (up to floating number precision) making it attractive from a computational point of view.

By a careful choice of the event times and deterministic evolution, it is possible to create and simulate an ergodic and non-reversible process with a desired unique invariant distribution (\cite{fearnhead2018}). The Zig-Zag  sampler (\cite{bierkens2019}) is a successful construction of such a processes. We now recap the intuition and the main steps behind the Zig-Zag  sampler.

The {\it one-dimensional} Zig-Zag  sampler is defined in the \textit{augmented space} $ (\xi, \theta) \in \mathbb{R} \times \{+1,-1\}$, where the first coordinate is viewed as the position of a moving particle and  the second coordinate as its velocity. The dynamics of the process $t\mapsto (\xi(t), \theta(t))$ (not to be confused with the time indexing the diffusion process) are as follows: starting from $(\xi(0), \theta(0))$, 
    \begin{enumerate}[label=(\alph*)]
        \item its flow is deterministic and linear in its first component with direction $\theta(0)$   and constant in its second component  until an event at time $\tau$ occurs. That is,  $\, (\xi(t), \theta(t)) = (\xi(0) + t \theta(0), \theta(0)), \, 0\le t\le \tau$.
        \item At an event time $\tau$, the  process changes the sign of its velocity, i.e.\  $(\xi(\tau), \theta(\tau)) = (\xi(\tau-),-\theta(\tau-))$. 
    \end{enumerate}
     The event times are simulated from an inhomogeneous Poisson process with specified rate  $\lambda\colon (\mathbb{R}\times \{1,-1\}) \rightarrow \mathbb{R}^+ $ such that $P(\tau \in [t, t + \epsilon] ) = \lambda(\xi(t),\theta(t)) \epsilon + o(\epsilon)$, $\epsilon \downarrow 0$.
     
     The {\it $d$-dimensional} Zig-Zag  sampler is conceived as the combination of $d$ one-dimensional Zig-Zag  samplers with rates $\lambda_i(\xi,\theta), \, i= 1,...,d$, where the rates create a coupling of the independent coordinate processes. The following result provides a sufficient condition for  the $d$-dimensional Zig-Zag  sampler to have a particular $d$-dimensional target density $\pi$ as invariant distribution.  Assume that the target $d$-dimensional distribution has strictly positive density with respect to the Lebesgue measure i.e.
        \[
        \pi(\dd\xi) \propto \exp(-\psi(\xi))\dd\xi,  \qquad   \xi \in \mathbb{R}^d.
        \]
        Define the \textit{flipping function} as $F_i(\theta) = (\theta_1,...,-\theta_i,...,\theta_d)$, for $\theta \in \{-1, +1\}^d$. For any $i = 1,...,d$ and $(\xi, \theta) \in \mathbb{R}^d \times \{ 1, -1 \}^d$, the Zig-Zag  process with Poisson rates satisfying
        \begin{equation}
            \label{condition stationary}
            \lambda_i(\xi,\theta) - \lambda_i(\xi,F_i(\theta)) = \theta_i \partial_{i} \psi(\xi), 
        \end{equation}
        has $\pi$ as invariant  density. Condition (\ref{condition stationary}) is derived in the supplementary material of \textcite{bierkens2019}.
    Condition (\ref{condition stationary}) is equivalent to
    \begin{equation}
        \label{Poisson rates}
        \lambda_i(\xi,\theta) = (\theta_i  \partial_{i} \psi(\xi))^+ + \gamma_i(\xi)
    \end{equation}
    for some $\gamma_i(\xi)\ge 0$. Throughout, we set $\gamma_i(\xi) = 0$ because generally the algorithm is more efficient for lower Poisson event intensity (see for example \cite{andrieu2019peskuntierney}, Subsection 5.4).    
    
Assume the target density is $\pi(\xi)=c\tilde\pi(\xi)$.     The process targets the specific distribution function through the Poisson rate $\lambda$ which is a function of the gradient of  $\xi\mapsto \psi(\xi) = -\log(\tilde{\pi}(\xi))$, so that any proportionality factor of the density disappears. Throughout we refer to the function $\psi$ as the \textit{energy function}. As opposed to standard Markov chain Monte Carlo methods, the process is not reversible and it is defined in continuous time. 
    
    \begin{exmp} Consider a $d$-dimensional Gaussian random variable with mean $\mu \in \mathbb{R}^d$ and positive definite covariance matrix $\Sigma \in \mathbb{R}^{d\times d} $. Then
    \begin{itemize}
    \item $\pi(\xi) \propto \exp\left(-(\xi - \mu)' \Sigma^{-1}(\xi - \mu)/2\right)$,
    \item $\partial_{k} \psi(\xi) = \left[\Sigma^{-1}(\xi - \mu)\right]_k$,
    \item $\lambda_k(\xi,\theta) = \left(\theta_k [\Sigma^{-1}(\xi - \mu)]_k \right)^+.$
    \end{itemize}
       Notice that if $\Sigma$ is  diagonal, then $\lambda_k(\xi, \theta) = 0$ whenever the process is directed towards the mean so that no jump occurs in the $k$th component when one of the following conditions is satisfied: $(\theta_k = -1, \xi_k-\mu_k \ge 0)$ or $(\theta_k = 1, \xi_k-\mu_k \le 0)$. In Figure~\ref{zigzag gaussian density} we simulate a realization of the Zig-Zag  sampler targeting a univariate standard normal random distribution.    
    \end{exmp}
    \begin{figure}[ht]
        \centering
        \includegraphics[ width=\linewidth]{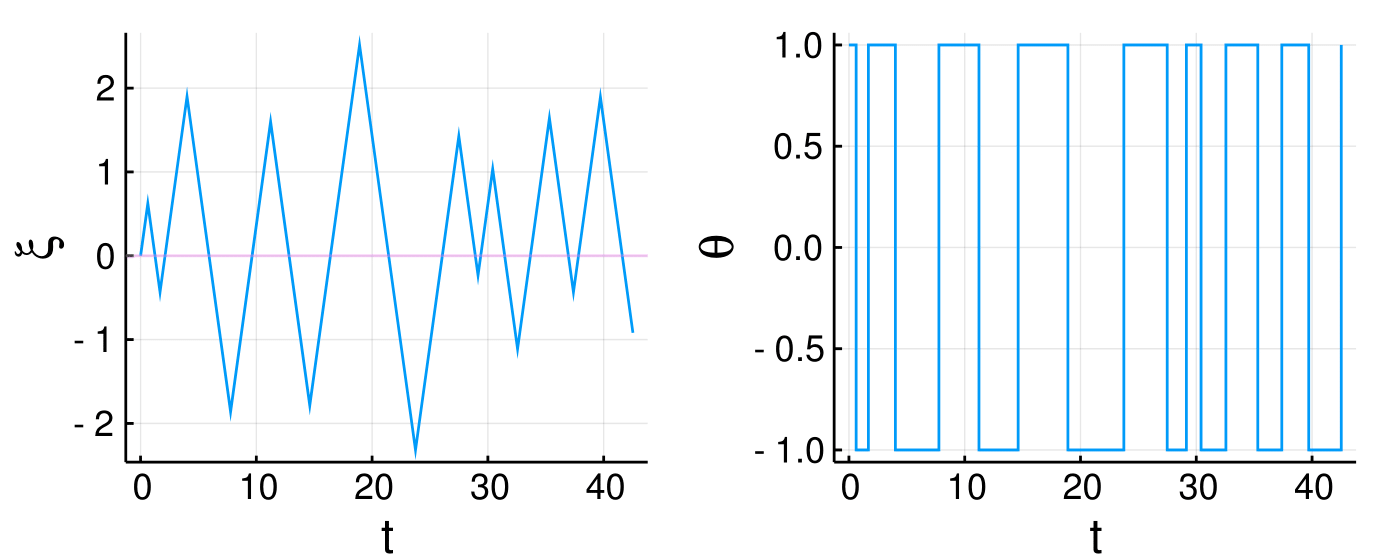}
        \caption{One dimensional Zig-Zag  targeting a Gaussian random variable $\mathcal{N}(0,1)$. Left: $t\mapsto \xi(t)$, right: $t\mapsto \theta(t)$.}
        \label{zigzag gaussian density}
    \end{figure}

Algorithm \ref{ZigZag1} shows the standard implementation of the Zig-Zag  sampler. Given a fixed time $t\ge0$ and a position $(\xi(t), \theta(t))$, the first event time $\tau^*$ after $t$ is determined by taking the minimum of event times $\tau_1, \tau_2,\dots,\tau_d$ simulated according to the Poisson rates $\lambda_i, i = 1,2,...,d$. At event time $\tau^*$, the velocity vector becomes $\theta(\tau^*) = F_{i^*}(\theta(t))$, with $i^* = \argmin(\tau_1,\dots,\tau_d)$.  The algorithm iterates this step moving forward each time until the next simulated event time  exceeds the final clock $\tau_{\text{final}}$.

Although we consider the velocities for each dimension of a $d$-dimensional Zig-Zag  process to be either $1$ or $-1$, these can be taken to be any non-zero values $(\theta_i, -\theta_i)$ for $i= 1,...,d$. A finetuning of $\theta_1,...,\theta_N$ can improve the performance of the sampler. Note that the only challenge in implementing Algorithm \ref{ZigZag1} lies on the simulation of the waiting times which correspond to the simulation of the first event time of $d$ inhomogeneous Poisson processes (IPPs) with rates $\lambda_1, \lambda_2,...,\lambda_d$ which are functions of the state space $(\xi, \theta)$ of the process. Since the flow of the process is linear and deterministic, the Poisson rates are known at each time and are equal to
\[
\lambda_i(t; \xi,\theta) = \lambda_i(\xi + t \theta, \theta), \qquad i = 1,2,...,d.
\]
To lighten the notation, we write $\lambda_i(t) := \lambda_i(t; \xi,\theta)$ when $\xi, \theta$ are fixed. Given an initial position $\xi$ and velocity $\theta$,  the waiting times $\tau_1,...,\tau_d$ are computed by finding the roots for $x$ of the equations 
\begin{equation}
    \label{find the root lambda}
    \int_0^x \lambda_i(s) \dd s + \log(u_i) = 0, \qquad i = 1,2,...,d,
\end{equation}
where $(u_i)_{i = 1,2,...,d}$ are independent realisations from the uniform distribution on $(0,1)$. When it is not possible to find roots of equation~\eqref{find the root lambda} efficiently, for example in closed form, it suffices to find upper bounds for the rate functions for which this is possible; Subsection~\ref{Sampling diffusion bridges} treats this problem for our particular setting. The linear evolution of the process and the jumps of the velocities are always trivially computed and implemented. 

Algorithm \ref{ZigZag1} returns a \textit{skeleton} of values corresponding to the position of the process at the event times. From these values, it is straightforward to reconstruct the continuous path of the Zig-Zag  sampler. Given a sample path of the Zig-Zag  sampler from 0 to $\tau_{\text{final}}$, we can obtain a sample from the target distribution in the following way: \begin{itemize}
    \item  Denote by $\xi(\tau)$ the value of the vector $\xi$ at the Zig-Zag  clock $\tau<\tau_{\text{final}}$. Fixing a sample frequency $\Delta\tau$, we can produce a sample from the density $\pi$ by taking the values of the random vector $\xi$ at time $\tau_{\text{burn-in}} + \Delta\tau, \tau_{\text{burn-in}} + 2\Delta\tau,...., \tau_{\text{final}}$ where $\tau_{\text{burn-in}}$ is the initial burn-in time taken to ensure that the process has reached its stationary regime. Throughout the paper, we create samples using this approach.
\end{itemize}

\begin{algorithm}[ht]
\begin{algorithmic}
\Procedure{ZigZag}{$\tau_{\text{final}}, \xi, \theta$} 
\State Initialise $k = 1 , \,t=0$
\State $\tau_j \sim$ IPP($\lambda_j(\cdot; \xi,\theta)$), \, $j = 1
,...,d$ \Comment{Draw from Inhomogeneous Poisson process (IPP)}
\While{$t\le \tau_{\text{final}}$}
\State $\tau^*, i^* \gets $ findmin($\tau_1,..,\tau_d$)
\State Update: $\xi \gets \xi + \theta(\tau^* - t)$ 
\State Update: $\theta_{i^*} \gets -\theta_{i^*};\quad t  \gets  \tau^*$
\State Save $\xi^{(k)} \gets \xi; \quad t^{(k)} \gets t$
\For{ $j = 1,...,d$}
\State $\tau_j \sim t +  \text{IPP}(\lambda_j(\cdot; \xi,\theta))$
\EndFor
\State $k \gets k + 1$
\EndWhile
\State \textbf{return} Skeletons $(\xi^{(l)}, t^{(l)})_{l=1,...,k-1}$
\EndProcedure
\end{algorithmic}
\caption{Standard $d$-dimensional Zig-Zag  sampler (\cite{bierkens2019})}
\label{ZigZag1}
\end{algorithm}
 
\subsection{Zig-Zag sampler for Brownian bridges}
\label{Zig-Zag sampler for Brownian bridges}
The previous subsections contain all ingredients necessary to run the Zig-Zag  sampler in a finite dimensional projection of the Brownian bridge measure $\mathbb{Q}^{0,v}$ on the interval $[0,T]$. We fix  $\bar\xi$ to  $v$ and run the Zig-Zag  sampler  for $\xi^N$ as defined in \eqref{eq:xiN} targeting a multivariate normal distribution.  
Figure~\ref{fig:brownian bridge} shows $100$ samples obtained from one sample run of the Zig-Zag  sampler where the coefficients are mapped to samples paths using \eqref{sum expansion}. The final clock of the Zig-Zag  is set to $\tau_{\text{final}} = 500$ with initial burning $\tau_{\text{burn-in}} = 10$.

Both Brownian motion and the Brownian bridge are special in that all coefficients in the Faber-Schauder  basis are independent. Of course, these processes can directly be simulated without need of a more advanced method like the Zig-Zag  sampler. However,  for a diffusion process with nonzero drift this property is lost. Nevertheless, we will see that when the process is expanded in the Faber-Schauder  basis, many coefficients  are still {\it conditionally} independent. This implies  that the dependency  graph of the joint density of the coefficients is sparse. We will show in  Section~\ref{A local Zig-Zag algorithm} how this property can be exploited efficiently using the Zig-Zag  sampler in its local version.

\begin{figure}[ht]
    \centering
    \includegraphics[width=1\linewidth]{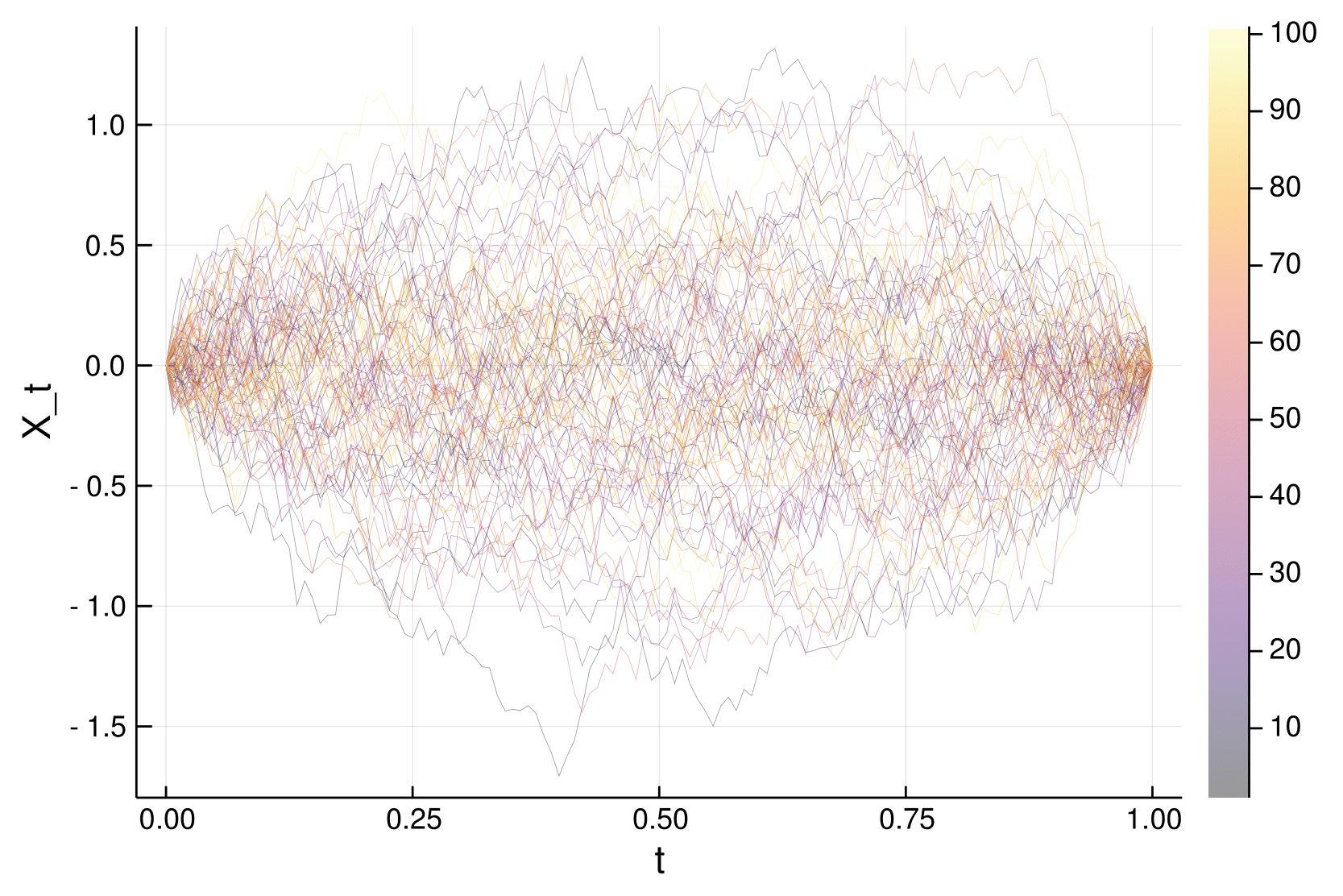}
    \caption{100 samples from the Brownian bridge measure starting at $0$ and hitting $0$ at time $1$ obtained by one run of the Zig-Zag  sampler targeting  the coefficients relative to the measure expanded with the Faber-Schauder  basis. The resolution level is fixed to $N = 6$ and the Zig-Zag  clock to $\tau_{\text{final}} = 500$ and initial burn in $\tau_{\text{burn-in}} = 10$.}
    \label{fig:brownian bridge}
\end{figure}

%% file: part4.tex
\section{Faber-Schauder expansion of diffusion processes}\label{faber schauder expansion of diffusion processes}
We extend the results of Section~\ref{preliminaries} to one-dimensional diffusions governed by the SDE in \eqref{sde}. Although the density is defined in infinite dimensional space, in this section we justify both intuitively and formally that the diffusion can be approximated to arbitrary precision by considering a finite dimensional projection of it. 
    
The intuition behind using the Faber-Schauder basis is that, under mild assumptions on the drift function $b$, any diffusion process behaves locally as a Brownian motion. Expanding the diffusion process with the Faber-Schauder functions, this notion translates to the existence of a level $N$ such that the random coefficients at higher levels which are associated to the Faber-Schauder basis are approximately independent standard normal and independent from $\xi^N$ under the measure $\mathbb{P}$. 

Define the function $Z_t\colon \mathbb{R}^+ \times C[0,T] \rightarrow \mathbb{R}^+$ given by
\begin{equation}
    \label{girsanov Z_1}
    Z_t(X) =  \exp\left( \int_0^t b(X_s) \dd X_s - \frac{1}{2}\int_0^t b^2(X_s)  \mathrm{d}s\right)
\end{equation}
where the first integral is understood in the  It\^o sense and  $X\equiv (X_s,\, s \in [0,T])$.
\begin{assumption}\label{A0} 
$Z_t$ is a $\mathbb{Q}$-martingale. 
\end{assumption}
 For sufficient conditions for verifying that this assumption applies, we refer to Remark~\ref{uniform growth}, Remark~\ref{Remark for logstic} and  \textcite{liptser2001statistics}, Chapter 6.
\begin{thm}\label{girsanov theorem}(Girsanov's theorem) If Assumption~\ref{A0} is satisfied, 
\begin{equation} \label{girsanov}
        \frac{\dd \PP^u}{\dd \QQ^u }(X) =  Z_T(X). 
\end{equation}
Moreover, a weak solution of the stochastic differential equation exists which is unique in law. 
\end{thm}
\begin{proof}
This is a standard result in stochastic calculus (see \cite{liptser2001statistics}, Section 6). 
\end{proof}
As we consider diffusions on $[0,T]$ with $T$ fixed, we denote $Z(X) := Z_T(X)$. Due to the appearance of the stochastic It\^o integral in $Z(X)$, we cannot substitute  for $X$ its truncated expansion in the Faber-Schauder basis. Clearly, whereas the approximation has finite quadratic variation, $X$ has not.
Assuming that $b$ is differentiable and applying  It\^o's lemma to the function $B(x) = \int_0^x b(s) \dd s$, the stochastic integral can be replaced and equation (\ref{girsanov Z_1}) is rewritten as
\begin{equation}
    \label{girsanov Z_2}
    Z(X) =  \exp \left( B(X_T) - B(X_0) - \frac{1}{2}\int_0^T \left( b^2(X_s) + b'(X_s)\right) \mathrm{d}s \right),
\end{equation}
 where $b'$ is the derivative of $b$.

\begin{defn}
Let $X$ be a diffusion governed by \eqref{sde}. Let $X^N$ be the process derived from $X$ by  setting to zero all coefficients of level exceeding $N$ in its Faber-Schauder expansion (see equation (\ref{sum expansion})). Set 
\[
Z^N(X) = \exp \left( B\left(X^N_T\right) - B\left(X^N_0\right) - \frac{1}{2}\int_0^T \left[ b^2\left(X^N_s\right) + b'\left(X^N_s\right)\right] \dd s \right).
\]
 We define the approximating measure $\mathbb{P}_N$ by the change of measure 
\begin{equation}
    \label{approximated stochastic process}
    \frac{\dd \mathbb{P}^u_N}{\dd \mathbb{Q}^u}(X) = \frac{Z^N(X)}{c_N},
\end{equation}
where 
 $c_N = \mathbb{E}_\mathbb{Q}\left(Z^N(X)\right)$.  
\end{defn}

Note that the measure $\mathbb{P}^u_N$ associated to the approximated stochastic process is still on an infinite dimensional space and such that the joint measure of random coefficients $\xi^N$ is different from the one under $\mathbb{Q}^u$ while the remaining coefficients stay independent standard normal and independent from  $\xi^N$. This is equivalent to approximating the diffusion process at finite dyadic points with Brownian noise fill-in in between every two points.  We now fix the final point $v_T$ by setting $\bar \xi = v_T$. Define the \textit{approximated stochastic bridge} with measure $\mathbb{P}^{u, v_T}_N$ in an analogous way of equation \eqref{approximated stochastic process}, so that  
\begin{equation}
    \label{approx diffusion bridge }
    \frac{\dd \mathbb{P}^{u, v_T}_N}{\dd \mathbb{Q}^{u, v_T}}(X) = \frac{Z^N(X)}{c^{v_T}_N}.
\end{equation}
where ${c^{v_T}_N} = \mathbb{E}_{\QQ^{u, v_T}}\left(Z^N(X)\right)$.
The following is the main assumption made.
\begin{assumption}\label{A1}
 The drift $b$ is continuously differentiable and $b^2 + b'$ is bounded from below.
\end{assumption}
\begin{thm}\label{weak convergence for bridges}
If Assumptions \ref{A0} and \ref{A1} are satisfied, then $\mathbb{P}^{u, v_T}_N$ converges weakly to $\mathbb{P}^{u, v_T}$ as $N \to \infty$.
\end{thm}
\begin{proof}
 In the following we lighten the notation by omitting the initial point $u$ from the notation, which will be assumed fixed to $u = x_0$. We wish to show that $\PP^{v_T}_N$ converges weakly to $\PP^{v_T}$ as $N \to \infty$. This is equivalent to showing that $\int f \dd \PP^{v_T}_N \to \int f \dd \PP^{v_T}$ for all bounded and continuous functions $f$. Write $c^{v_T}_\infty = p(0,x_0,T, v_T)/q(0,x_0,T, v_T)$. By equation~\eqref{bayes rule bridge} and~\eqref{girsanov}, 
 \[ \EE_{\QQ^{v_T}} Z(X) = \EE_{\QQ^{v_T}} \frac{d \PP^{x_0}}{d \QQ^{x_0}} = c_{\infty}^{v_T} \EE_{\QQ^{v_T}} \left[ \frac{ d \PP^{v_T}}{d \QQ^{v_T}}\right] = c_{\infty}^{v_T}\] and  we have that
\begin{align}\label{eq:boundZ}
& 	\left| \int f \dd \PP^{v_T}_N - \int f \dd \PP^{v_T}\right| \nonumber \\
& = \left| \int f \left( \frac{Z^N}{c^{v_T}_N} - \frac{Z}{c^{v_T}_\infty} \right) \dd \QQ^{v_T} \right|\nonumber \\
	& \le \|f\|_\infty \int \left| \frac{Z^N(X)}{c^{v_T}_N} - \frac{Z(X)}{c^{v_T}_\infty}\right| \dd \QQ^{v_T}(X)\nonumber \\
	& \leq   \|f\|_\infty \left( \frac 1 {c^{v_T}_N} \int  \left|Z^N(X)-Z(X)\right|\dd \QQ^{v_T}(X) + \int Z(X)  \left|\frac {1}{c^{v_T}_N} - \frac {1}{c^{v_T}_{\infty}}  \right|   \dd \QQ^{v_T}(X) \right) \nonumber \\
	& \le \|f\|_\infty \left(\frac{1}{c^{v_T}_N} \int \left|Z^N(X)-Z(X)\right|\dd \QQ^{v_T}(X) + \left| \frac{c^{v_T}_\infty}{c^{v_T}_N}-1 \right|\right) 
\end{align}
where we used Assumption~\ref{A0} for applying the change of measure between the conditional measures.  
Notice that $Z^N(X) = Z(X^N)$. The mapping $X \mapsto Z(X)$, as a function acting on $C(0,T)$ with uniform norm, is continuous, since $B$, $b$, and $b'$ are continuous. 
Therefore, it follows from the L\'evy-Ciesielski construction of Brownian motion (see Section~\ref{sec:faber-schauder}) and the continuous mapping theorem that
\[Z^N(X) \rightarrow Z(X) \qquad  \mathbb{Q}^{v_T}-a.s.
\]
Now notice that, under conditional measures $\QQ^{v_T}$ and $\PP^{v_T}$, the term $B(X_T) - B(X_0)$ is fixed. By the assumptions on $b$ and $b'$, $Z$ is a bounded function and by dominated convergence we get that 
\[
\lim_{N \rightarrow \infty} \EE_\QQ^{v_T} |Z^N(X)-Z(X)| = 0
\]
giving convergence to zero of the first term in~\eqref{eq:boundZ}. This implies that also the constant $c_N := \EE_\QQ^{v_T} |Z^N(X)| $ converges to $\EE_\QQ^{v_T} |Z(X)| = c^{v_T}_\infty $  so that all the terms in (\ref{eq:boundZ}) converge to 0.  
\end{proof}

We now list some technical conditions for the process to satisfy Assumptions \ref{A0} and \ref{A1}.  

\begin{rmk}\label{uniform growth}
If $|b(x)| \le c(1 + |x|)$, for some positive constant $c$, then Assumption~\ref{A0} is satisfied. 
\end{rmk}
\begin{proof}
See \textcite{liptser2001statistics}, Section 6, Example 3 (b).
\end{proof}
\begin{rmk}\label{remark gloabl lipschitz}
If $b$ is globally Lipschitz and continuously differentiable, then Assumptions~\ref{A0} and \ref{A1} are satisfied.
\end{rmk}
\begin{proof}
Assumption~\ref{A1} is trivially satisfied. By Remark \ref{uniform growth}, also Assumption~\ref{A0} is satisfied. 
\end{proof}
In Subsection~\ref{Diffusion model with unbounded drift} we will present an example where the drift $b$ is not globally Lipschitz, yet Assumption~\ref{A1} is satisfied. 

\begin{assumption}\label{Assumptions on B}
 There exists a non-decreasing function $h :[0,\infty) \rightarrow [0,\infty)$ such that ${B(x) \leq h(|x|)}$ and
\[ \int_0^{\infty} \exp(h(x) - x^2/(2T)) \, d x < \infty. \]
\end{assumption}
The above integrability condition is for example satisfied if $h(|x|) = c(1 + |x|)$ for some $c > 0$.
\begin{rmk}\label{Remark for logstic}
If Assumptions \ref{A1} and \ref{Assumptions on B} hold, then Assumption~\ref{A0} is satisfied. 
\end{rmk}

\begin{proof}
By Subsection 3.5 in \textcite{karatzas1998brownian}, $(Z_t)$ is a local martingale. Say $b'(x) + b^2(x) \geq -2  C$, where $C \geq 0$. Using the assumptions, we have
\[ Z_t = \exp \left( B(X_t) - B(X_0) - \tfrac 1 2 \int_0^t \{ b'(X_s) + b^2(X_s) \} \, ds \right) \leq A\exp(C t) \exp(h(|X_t|)),\]
with constant $A = \exp(-B(X_0))$. Then
\[ \sup_{t \in [0,T]} Z_t \leq A\sup_{t \in [0,T]} \exp(C t) \exp(h(|X_t|)) \leq A\exp(C T) \exp\left(h \left(\max_{t \in [0,T]}| X_t|\right)\right).\]
By Lemma~\ref{lemma integrability}, below 
\[ \mathbb E \sup_{t \in [0,T]} Z_t \leq A \exp(C T)\, \mathbb E \exp(h (\max_{t \in [0,T]}| X_t|)) < \infty.\]
Then for a sequence of stopping times $(\tau_k)$ diverging to infinity such that $(Z_t^{\tau_k})_{0 \leq t \leq T}$ is a martingale for all $k$, we have
\[ \EE Z_0 = \EE Z^{\tau_k}_0  = \EE Z^{\tau_k}_t \rightarrow \EE Z_t \]
as $k \rightarrow \infty$ by dominated convergence.
\end{proof}

\begin{lem}\label{lemma integrability}
Suppose $h\colon [0,\infty) \rightarrow [0,\infty)$ is non-decreasing. Let $N_T = \max_{0 \leq t \leq T} |X_t|$ where $(X_t)$ is a Brownian motion. Then
\[ \mathbb E \exp h(N_T) \leq 4 \int_0^{\infty} \frac 1 {\sqrt{2 \pi T}} \exp(h(x) - x^2/(2T)) \, d x.\]
\end{lem}
\begin{proof}
 The maximum $M_T = \max_{0 \leq t \leq T} X_t$ of a Brownian motion is distributed as the absolute value of a Brownian motion and thus has density function $\frac 2 {\sqrt{2 \pi T}} \exp(-x^2/(2T))$, see \textcite{karatzas1998brownian}, Subsection 2.8.
 We have
 $\mathbb P(N_T \geq y) \leq 2 \mathbb P(M_T \geq y)$ from which the result follows.
\end{proof}
Finally we mention that Theorem~\ref{weak convergence for bridges} can be generalized in the following
way to diffusions without a fixed end point.
\begin{prop}\label{main theorem}
If Assumption~\ref{A1} is satisfied and $B$ is bounded, then $\mathbb{P}_N$ converges weakly to $\mathbb{P}$.
\end{prop}
\noindent The proof follows the same steps of the one of Theorem~\ref{weak convergence for bridges}. In this case we need to pay attention on $B$, as for unconditioned process, the final point is not fixed. If $B$ is bounded, then Assumption  \ref{Assumptions on B} is satisfied. By Remark \ref{Remark for logstic} also Assumption~\ref{A0} is satisfied so that we can apply Theorem~\ref{girsanov theorem} for the change of measure. Finally, by the assumptions on $b$ and $B$, the function  $Z$ is bounded and by dominated convergence we get that
\[
\lim_{N \rightarrow \infty} \EE_\QQ |Z^N(X)-Z(X)| = 0.
\]

%% file: part5.tex
\section[A local Zig-Zag algorithm]{A local Zig-Zag algorithm with subsampling for \\high-dimensional structured target densities
} 
\label{A local Zig-Zag algorithm}
In Subsection~\ref{Sampling diffusion bridges} we will show that the task of sampling diffusion bridges boils down to the task of sampling a high-dimensional vector $\xi^N \in \RR^{M}$ under the measure $\mathbb{P}^{u,v_T}_N$.
Define by $P_{\xi^N}$ the distribution of the vector $\xi^N$. Under the target measure,
\[
P_{\xi^N}(\dd \xi^N) = \pi(\xi^N) \dd \xi^N.
\]
We take the density $\pi$ to be the $M$-dimensional invariant density (target density) for the Zig-Zag sampler. An efficient implementation of piecewise deterministic Monte Carlo methods, including the local and fully local Zig-Zag sampler can be found in \textcite{mschauer/ZigZagBoomerang.jl}.

\subsection{Subsampling technique}\label{sub-sampling technique sub section}
In our setting, the integral appearing in the Girsanov formula \eqref{girsanov Z_2} poses difficulties when finding the root of equation \eqref{find the root lambda} and would require numerical evaluation of the integral, hence also introducing a bias. By adapting the subsampling technique presented in \textcite{bierkens2019} (Section 4) we avoid this problem altogether (see Subsection~\ref{Sampling diffusion bridges}). In general this technique requires 
\begin{enumerate}[label=(\alph*)]
    \item unbiased estimators for $\partial_i\psi$ i.e. random functions $\partial_i\Tilde{\psi_i}(\xi, U_i)$ such that \[E_{U_i}[\partial_i\Tilde{\psi_i}(\xi, U_i)] = \partial_i\psi(\xi),\] for all  $i$ and $\xi$. These random functions create new (random) Poisson rates given by 
    \begin{equation}
        \label{unbiased estimator}
        \Tilde{\lambda}_i(t; \xi, \theta; U_i) = (\theta_i \partial_i \Tilde{\psi}(\xi(t), U_i))^+, \qquad i = 1,2,...,d, 
    \end{equation}
    whose evaluation becomes feasible and computationally more efficient compared to the original Poisson rates given by equation (\ref{Poisson rates}).
    \item upper bounds $\bar\lambda_i:(\RR^+ \times \RR^d \times \{-1,+1\}^d) \rightarrow \RR^+$ for all $i = 1,...,d$ such that for any point $(\xi, \theta)$ and $t\ge 0$ we have
    \begin{equation}
        \label{condition for upper bounds}
            P\left(\Tilde{\lambda}_i(t; \xi, \theta; U_i)\le \Bar{\lambda}_i(t; \xi, \theta)\right) = 1.
    \end{equation}
    As we show in Algorithm \ref{ZigZagsubsampler} and in Section~\ref{Numerical Results}, these upper bounds are used for finding the roots of equation (\ref{find the root lambda}).    
\end{enumerate}
Algorithm \ref{ZigZagsubsampler} gives the algorithm for the Zig-Zag sampler with subsampling. It can be proved (see \cite{bierkens2019}) that the Zig-Zag sampler with subsampling has the same invariant distribution as its original and therefore does not introduce any bias. Note that we  slightly modified the algorithm from \textcite{bierkens2019} in order to reduce its complexity. In particular it is sufficient to draw new waiting times and to save the coordinates only when the \textit{if} condition at the \textit{subsampling step} of Algorithm \ref{ZigZagsubsampler} is true.

\begin{algorithm}[ht]
\caption{$d$-dimensional Zig-Zag sampler with subsampling}
\label{ZigZagsubsampler}
\begin{algorithmic}
\Procedure{ZigZag\_ws}{$\tau_{\text{final}}, \xi, \theta$} 
\State Initialise $k = 1 , \, t=0$
\State $\tau_j \sim$ IPP($\bar \lambda_j(\cdot;\xi,\theta)$), $j = 1
,...,d$
\While{$t\le \tau_{\text{final}} $}
\State $\tau^*, i^* \gets $ findmin($\tau_1,...,\tau_d$)
\State $\xi^{old} \gets \xi$ 
\State Update: $\xi \gets \xi + \theta(\tau^* - t)$
\State Update: $\Delta t \gets \tau^* - t; \quad  t  \gets \tau^*$
\State $U_{i^*} \sim \text{Law}(U_{i^*}), V \sim \operatorname{Unif}(0,1)$
\If{$V \le \Tilde{\lambda}_{i^*}(0, \xi, \theta, U_{i^*})/ \Bar{\lambda}_{i^*}(\Delta t; \xi^{old}, \theta)$} 
\Comment{Subsampling step}
    \State Save $\xi^{(k)} \gets \xi,\quad t^{(k)}\gets t$
    \State $k \gets k + 1$
    \State $\theta_{i^*} \gets -\theta_{i^*}$
    \For{$j \in \{1,\dots,d\}\setminus \{i^*\}$}
    \State $\tau_j \sim t + \text{IPP}(\Bar{\lambda}_j(\cdot; \xi,\theta))$
    \EndFor
\Else
    \State $\tau_{i^*} \sim t +\text{IPP}(\Bar{\lambda}_{i^*}(\cdot; \xi,\theta))$
\EndIf
\EndWhile
\State \textbf{return} Skeletons $(\xi^{(l)}, t^{(l)})_{l=1,2,...,k-1}$
\EndProcedure
\end{algorithmic}
\end{algorithm}

\subsection{Local Zig-Zag sampler}
\label{Subsec: Local ZigZag sampler section}
Subsection 3.1 of \textcite{2015arXiv151002451B} proposes a local algorithm for the \textit{Bouncy Particle Sampler} which is a process belonging to the class of piecewise deterministic Markov processes. Similar ideas apply to our setting. 

\begin{assumption}\label{Assumption dependecy structure}
 The Poisson rate $\lambda_i$ for a $d$-dimensional target distribution is a function of the coordinates $N_i \subset \{1, \dots, d\}$, 
 $$\lambda_i(s; \xi, \theta) = \lambda_i(s; \xi_k, \theta_k : k \in N_i).$$
\end{assumption}

Recall that by the definition of $\lambda_i$ (see equation \eqref{Poisson rates}), the $i$th partial derivative of the negative loglikelihood determines the sets $N_i$. Now let us suppose that the first event time $\tau$ is triggered by the coordinate $i$ so that at event time, the velocity $\theta_i$ is flipped. For all $\lambda_k$ which are not function of this coordinate ($k \not\in N_i$), we have
$$
\lambda_k^{old}(\tau + s) = \lambda_k^{new}(s),
$$
which implies that the waiting times drawn before $\tau$, are still valid after switching the velocity $i$. This allows us to rescale the previous waiting time and reduce the number of computations at each step. The sets $N_{1},...,N_{d}$ are connected to the factorisation of the target distribution and define its conditional dependence structure. Indeed, take a $d$-dimensional target distribution with the following decomposition 
\[
\pi(\xi) = \prod_{i = 1}^N \pi_i(\xi^{(i)})
\]
where $\xi^{(i)} := \{ \xi_j: j \in \Gamma_i\}$ and  $\Gamma_i \subset \{ 1,2,...,N\}$ defines a subset of indices. We have that 
\[
-\partial_k \log(\pi(\xi)) = -\sum_{i =1}^N \partial_k \log \pi_i(\xi^{(i)}), \quad k = 1,...,d 
\] 
where the $i$th term in the sum is equal to 0 if $k \notin \Gamma_i$. Since the Poisson rates \eqref{Poisson rates} are defined through the partial derivatives, the factorisation defines the sets $N_1,...,N_d$ of Assumption~\ref{Assumption dependecy structure}.

Algorithm \ref{ZigZaglocal} shows the implementation of the local sampler which exploits any conditional independence structure so that
the complexity of the algorithm scales well with the number of dimensions. 

The local Zig-Zag sampler simplifies to independent one-dimensional Zig-Zag processes if the coefficients are pairwise independent coefficients, as it was the case in the example of sampling a Brownian motion or Brownian bridge (see Subsection~\ref{Zig-Zag sampler for Brownian bridges}).
On the other hand, it defaults to Algorithm \ref{ZigZag1} when the dependency graph is fully connected, that is if $N_i = \{1, \dots, d\}, \forall i$.

\begin{algorithm}[ht]
\begin{algorithmic}
\State \emph{Input}: The bounds $\bar\lambda_i$ depend only on $\xi_k, \theta_k,$ for $k \in N_i$
\Procedure{ZigZag\_local}{$\tau_{\text{final}}, \xi, \theta$}
\State Initialise $k = 1 , \, t=0$
\State $\tau_j\sim$ IPP($ \lambda_j(\cdot; \xi,\theta)$), $j = 1
,...,d$ 
\While{$t\le \tau_{\text{final}} $}
\State $\tau^*, i^* \gets$ findmin$(\tau_1,...,\tau_d)$ 
\State Update: $\xi \gets \xi + \theta (\tau^* - t)$
\State Update: $\theta_{i^*} \gets -\theta_{i^*}; \quad t  \gets \tau^*$
\State Save $\xi^{(k)} \gets \xi;\quad t^{(k)} \gets t$
\State $k \gets k + 1$
\For{$j$ in $N_{i^*}$}
    \Comment{Local step}
    \State $\tau_j \sim t +\text{IPP}(\lambda_j(\cdot; \xi,\theta))$ 
\EndFor
\EndWhile
\State \textbf{return} Skeletons $(\xi^{(l)}, t^{(l)})_{l=1,...,k-1}$
\EndProcedure
\end{algorithmic}
\caption{$d$-dimensional local Zig-Zag sampler}
\label{ZigZaglocal}
\end{algorithm}

\begin{algorithm}[ht]
\begin{algorithmic}
\State \emph{Input}: The bounds $\bar\lambda_i$ depend only on $\xi_k, \theta_k,$ for $k \in \bar N_i$ and the random Poisson rates $\tilde\lambda_i$ (eq. \eqref{unbiased estimator}) depends only on $U_i$ (the randomizing argument of $\tilde\partial_i\psi$) and $\xi_k, \theta_k$ for $k \in \tilde N_i(U_i)$
\Procedure{ZigZag\_fully\_local}{$\tau_{\text{final}}, \xi, \theta$}
\State Initialise: $k = 1$, $t = \mathbf{0} \in \mathbb{R}^d$, $\tau^* = 0$ 
\State  $\tau_j \sim$ IPP($\bar\lambda_j(\cdot; \xi ,\theta$)), $j = 1,...,d$ 
\While{$\max(t)\le \tau_{\text{final}} $}
\State $\tau_{i^*}^{old} \gets \tau^*, \quad \xi^{old}_{i^*} \gets \xi_{i^*}$
\State $\tau^*, i^* \gets$ findmin$(\tau_1,...,\tau_d)$
\State $U_{i^*} \sim \operatorname{Law}(U_{i^*})$
\For{$j$ in $\bar N_{i^*} \cup \tilde N_{i^*}(U_{i^*})$}
    \State Update: $\xi_j \gets \xi_j + \theta_j (\tau^* - t_j)$
    \State Update: $t_j  \gets  \tau^*$
\EndFor
\State{$V \sim \operatorname{Unif}(0,1)$}
\If{$V \le \Tilde{\lambda}_{i^*}(0; \xi, \theta; U_{i^*})/ \Bar{\lambda}_{i^*}(\tau^* - \tau_{i^*}^{old}; \xi^{old}, \theta)$} 

    \State Update: $\theta_{i^*} \gets -\theta_{i^*}$
    \State Update: $k \gets k + 1$
    \State Save: $i^{(k)} \gets i^*, s^{(k)} \gets \tau^*, \xi^{(k)} \gets \xi_{i^*}$
    \For{$n$ in $\left(\bigcup_{j \in  \bar N_{i^*}} \bar N_{j}\right) \!\setminus\! \left(\bar N_{i^*} \cup \tilde N_{i^*}(U_{i^*})\right)$}
        \State Update: $\xi_n \gets \xi_n + \theta_n (\tau^* - t_n)$
        \State Update: $t_n  \gets  \tau^*$
    \EndFor
    \For{$j$ in $\bar N_{i^*}\!\setminus\!\{i^*\}$}
    \State  $\tau_j \sim \tau^* + \operatorname{IPP}(\Bar{\lambda}_j(\cdot; \xi,\theta))$
    \State $\tau_j^{old} \gets \tau^*, \quad \xi^{old}_j \gets \xi_j$
    \EndFor
\EndIf
\State  $\tau_{i^*} \sim \tau^* + \operatorname{IPP}(\Bar{\lambda}_{i^*}(\cdot;\xi,\theta))$
\EndWhile
\State \textbf{return} reflection tuples $((i^{(l)}, s^{(l)}, \xi^{(l)}))_{l= 1,...,k}$
\EndProcedure
\end{algorithmic}
\caption{Implementation of the $d$-dimensional fully local Zig-Zag sampler}
\label{alg: ZigZaglDoublyLocal}
\end{algorithm}

\subsection{Fully local Zig-Zag sampler}
Combining the subsampling technique and the local ZZ can lead to a further reduction of the complexity of the algorithm. Indeed the bounds for the Poisson rates might induce sparsity as  $\Bar{\lambda}_i$ can be function of few coordinates (see for example Subsection~\ref{sin model}). This means that, after flipping $\theta_i$, $\Bar{\lambda}^{old}_j(\tau + t) = \Bar{\lambda}^{new}_j(t) $
for almost all $j \ne i$ making the \text{if} statement in the \textit{local step} of Algorithm \ref{ZigZaglocal} almost always satisfied and improving the efficiency of the algorithm.  This means that, after flipping $\theta_i$, we have that $\Bar{\lambda}^{old}_j(\tau + t) = \Bar{\lambda}^{new}_j(t) $
for almost all $j \ne i$ or, in other words, the cardinality of the set $N_i$ in the \textit{local step} of Algorithm \ref{ZigZaglocal} is small. Furthermore, the evaluation of $\tilde \lambda_i(t, \xi, \theta)$ and  $\bar \lambda_i(t, \xi, \theta)$ for $i = 1,2,...,d$ does not necessarily require to access the location of all the coordinates $\xi_j$ so that, by assigning an independent time for each coordinate and updating only the coordinates needed for the evaluation of $\tilde \lambda_i$ and $\bar \lambda_i$, the algorithm can be made more efficient. This is shown in the fully local ZZ sampler (Algorithm \ref{alg: ZigZaglDoublyLocal}) where $\bar N_i, \tilde N_i(U_i)$ define respectively the subset and the random subset of the coordinates required for the evaluation of $\bar \lambda_i(\cdot; \xi, \theta)$ and $\tilde\lambda_i(\cdot; \xi, \theta; U_i)$. 

\subsection{Sampling diffusion bridges}
\label{Sampling diffusion bridges}

In order to employ the Zig-Zag sampler to simulate from the bridge measure we choose the truncation level $N$ in equation (\ref{sum expansion}). Then, under $\PP^{u, v_T}_N$
\[
 \pi(\dd \xi^N) \propto  Z^N(X)  \exp\left( \frac{-\|\xi^N\|^2}{2}\right) \dd \xi^N.
\]
This is a straightforward consequence of the change of measure in \eqref{approx diffusion bridge }  and the L\'evy-Ciesielski construction.

We need to make one further assumption:
\begin{assumption}\label{A2}
The drift $b$ of the diffusion process is twice differentiable.
\end{assumption}
Assumption~\ref{A2} is necessary in order to compute the $\xi_k$-partial derivative of the energy function, which becomes
\begin{equation} \label{partial}
\partial_k \psi(\xi^N) = \frac12\int_{S_k} h_k(s; \xi^N) \dd s + \xi_k, 
\end{equation}
where 
\[ h_k(s; \xi^N) =  \phi_k(s)\left(2b(X^{N}_s)b'(X^{N}_s) + b''(X^{N}_s)\right). \]
As the index $k$ in the Faber-Schauder basis function gets larger, both the magnitude of $\phi_k$ and the size of its support  decrease so that typically $\int h_k(s; \xi^N) \dd s$ gets smaller and $\partial_k \psi(\xi) \approx \xi_k$ which corresponds to the partial derivative of the energy function of a standardized Gaussian random variable with independent components. This justifies one more time the intuition that for high levels $i$, the random variables $\xi_{ij}$, $j =1 ,...,2^{i}-1$ are approximately normally distributed and almost independent from the other random coefficients.

In order to avoid the evaluation of the integral appearing in (\ref{partial}) and the difficulty of drawing a Poisson time from its corresponding rate (\ref{Poisson rates}), we employ the subsampling technique. Considering $\xi^N$ nonrandom, we take as an unbiased estimator for $\partial_k\psi(\xi_N)$ the (random) function 
\begin{equation}
    \label{eq: unbiased Poisson estimator}
    \frac12|S_k| h_k(U_k; \xi^N)  + \xi_k, 
\end{equation}
where $U_k \sim \text{Unif}(S_k)$ and as the bounding intensity rate
\begin{equation}
    \label{eq: bounding Poisson rate}
    \bar \lambda_k(t, \xi^N,  \theta^N) = \frac 12 |S_k||\theta_k|\Bar\Phi_k f(\xi^N(t)) + \left(\theta_k \xi_k(t)\right)^+ , \quad \xi^N \in \RR^{M},
\end{equation}
where  $\Bar\Phi_k = \max_s(\phi_k(s))$ and $f(\xi^{N}) \ge  \left|2b(X^{\xi^N}_s)b'(X^{\xi^N}_s) + b''(X^{\xi^N}_s)\right|,\, \forall s \in [0,T],\, \xi^N \in \RR^{M} $. The subsampling technique avoids the numerical computation of the time integral~\eqref{partial}, thus avoiding a numerical bias and reducing the computational effort from $\mathcal O(T)$ (for fixed discretization size) to $\mathcal O(1)$.
The variance of this unbiased estimator can be reduced by averaging over multiple independent uniform draws or similar strategies (see for example Section~\ref{Subsection: Numerical comparisons}), albeit  at the cost of additional computations.  In Section~\ref{Numerical Results} we show specifically for each numerical experiment how we derived the Poisson upper bounds $\Bar \lambda_i$. 

The compact support of the Faber-Schauder functions induce a sparse dependency structure on the target measure $\pi$. Indeed, $X_t$ only depends on those values of $\xi_{l,k}$ for which $t \in S_{l,k}$. See Figure~\ref{Fig: support of the fs functions} for an illustration. It is easy to see that this implies that $\frac{\partial \psi(\xi^N)}{\partial \xi_{(i,j)}} $ depends only on those $\xi_{(k,l)}$ for which the interior of $S_{i,j} \cap S_{k,l}$ is non-empty. In particular, define the relation $ \xi_{i,j} \ll \xi_{k,l}$ to hold if $S_{k,l} \subset S_{i,j}$. If this happens, then we refer to $\xi_{i,j}$ as the \emph{ancestor} of $\xi_{k,l}$ (and conversely $\xi_{k,l}$ as the \emph{descendant}). 
  Then the sets in Assumption~\ref{Assumption dependecy structure} (using double indexing) can be chosen as $N_{i,j} = \{ \xi_{h,d} \colon \xi_{h,d} \ll \xi_{i,j} \vee  \xi_{h,d} \gg \xi_{i,j}\}$ with cardinality $|N_{i,j}|= 2^{N-i + 1} + i -1 $, where $N$ is the truncation level. Formally, $N_{i,j}$ are the neighborhoods of the interval graph induced by $((S_{i,j}\colon i \in \{1,2,\dots,N\},\, j \in \{0,1,\dots,2^i-1\}))$ with vertices $\{(i,j) \colon i \in \{1,2,\dots,N\},\, j \in \{0,1,\dots,2^i-1\}\}$, where there is an edge between $(i,j)$ and $(l,k)$ if the interior of $S_{i,j} \cap S_{k,l}$ is non-empty (see Figure~\ref{depepndence graph}). The factorization of the partial derivatives leads to a specific dependency structure of the coefficients under the target diffusion bridge measure: the coefficient $\xi_{i,j}$ is conditionally independent of the coefficient $\xi_{k,l}$ if $S_{i,j} \cap S_{k,l} = \emptyset$ conditionally on the set of common ancestors $(\xi_{m,n} \colon \xi_{m,n} \ll \xi_{i,j} \wedge  \xi_{m,n} \ll \xi_{k,l})$. This argument is made more formal by decomposing the likelihood function in Appendix~\ref{App: fact diff bridge}.

\input{fig4_2.tex}

%% file: fig4_2.tex
\tikzset{main node/.style={circle,fill=blue!10, draw, minimum size=0.7cm,inner sep=0pt},
}
\tikzset{>=latex}         

\begin{figure}
    \centering

\begin{tikzpicture}
	\begin{pgfonlayer}{nodelayer}
	    \node [style=none] (A) at (-4, -1.5) {};
	    \node [style=none] (B) at (4, -1.5) {};
		\node [style=none] (0) at (-4, 0) {};
		\node [style=none] (1) at (-3, 0) {};
		\node [style=none] (2) at (-2, 0) {};
		\node [style=none] (3) at (-1, 0) {};
		\node [style=none] (4) at (0, 0) {};
		\node [style=none] (5) at (1, 0) {};
		\node [style=none] (6) at (2, 0) {};
		\node [style=none] (7) at (3, 0) {};
		\node [style=none] (8) at (4, 0) {};
		\node [style=none] (9) at (-4, 1) {};
		\node [style=none] (10) at (-2, 1) {};
		\node [style=none] (11) at (0, 1) {};
		\node [style=none] (12) at (2, 1) {};
		\node [style=none] (13) at (4, 1) {};
		\node [style=none] (14) at (-4, 2) {};
		\node [style=none] (15) at (0, 2) {};
		\node [style=none] (16) at (4, 2) {};
		\node [style=none] (17) at (-4, 3) {};
		\node [style=none] (18) at (4, 3) {};
	\end{pgfonlayer}
	\begin{pgfonlayer}{edgelayer}
	    \draw [Circle-Stealth] (A)node[below] {$0$}  -- (B) node[below] {$T$};
		\draw [{Bracket[width=1.0em]}-{Bracket[width=1.0em]}] (0) to node[below] {$S_{3,0}$} (1) ;
		\draw [{Bracket[width=1.0em]}-{Bracket[width=1.0em]}] (1) to node[below] {$S_{3,1}$} (2);
		\draw [{Bracket[width=1.0em]}-{Bracket[width=1.0em]}] (2) to node[below] {$S_{3,2}$} (3);
		\draw [{Bracket[width=1.0em]}-{Bracket[width=1.0em]}] (3) to node[below] {$S_{3,3}$}  (4);
		\draw [{Bracket[width=1.0em]}-{Bracket[width=1.0em]}] (4) to node[below] {$S_{3,4}$}  (5);
		\draw [{Bracket[width=1.0em]}-{Bracket[width=1.0em]}] (5) to node[below] {$S_{3,5}$}  (6);
		\draw [{Bracket[width=1.0em]}-{Bracket[width=1.0em]}] (6) to node[below] {$S_{3,6}$}  (7);
		\draw [{Bracket[width=1.0em]}-{Bracket[width=1.0em]}] (7) to node[below] {$S_{3,7}$}  (8) ;
		\draw [{Bracket[width=1.0em]}-{Bracket[width=1.0em]}] (9) to node[below] {$S_{2,0}$}  (10);
		\draw [{Bracket[width=1.0em]}-{Bracket[width=1.0em]}] (10) to node[below] {$S_{2,1}$} (11);
		\draw [{Bracket[width=1.0em]}-{Bracket[width=1.0em]}] (11) to node[below] {$S_{2,2}$}(12);
		\draw [{Bracket[width=1.0em]}-{Bracket[width=1.0em]}] (12) to node[below] {$S_{2,3}$}(13);
		\draw [{Bracket[width=1.0em]}-{Bracket[width=1.0em]}] (14) to node[below] {$S_{1,0}$}(15);
		\draw [{Bracket[width=1.0em]}-{Bracket[width=1.0em]}] (15) to node[below] {$S_{1,1}$}(16);
		\draw [{Bracket[width=1.0em]}-{Bracket[width=1.0em]}] (17) to node[below] {$S_{0,0}$} (18);
	\end{pgfonlayer}
\end{tikzpicture}
    \caption{Support of the Faber-Schauder functions $(\phi_{i,j} : i \in \{0,1,\dots,N\}, \, j = \{0,1,\dots,2^i-1\}$ with $N =3$. The coefficient $\xi_{i,j}$ is  independent of the coefficient $\xi_{k,l}$  conditionally on the set of common ancestors $(\xi_{m,n} \colon S_{m,n} \cap S_{i,j}\ne \emptyset \wedge S_{m,n} \cap S_{k,l} \ne \emptyset)$ if $S_{i,j} \cap S_{k,l} = \emptyset$.}
    \label{Fig: support of the fs functions}
\end{figure}
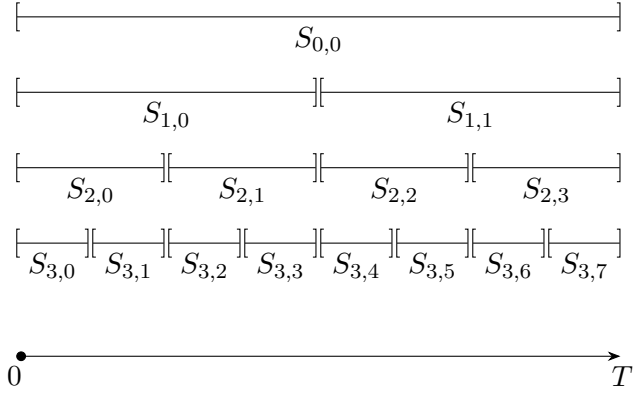

%% file: part6.tex
\section{Numerical results}\label{Numerical Results}
We show numerical results for three representative examples. 
 In general, when applying our method, we start from a model \eqref{sde}, devise a representation of the approximate diffusion bridge 
\eqref{approx diffusion bridge }, that we 
sample using generic implementations of algorithms 1-4 from our package, which are easily adapted to the task of sampling the coefficients of the Faber-Schauder expansion.  To this end, we provide the $k$-th partial derivative of the energy function \eqref{partial} or an upper bound to the Poisson rate \eqref{eq: bounding Poisson rate} as argument for the sampler, as well as the sets $N_{i,j}$ as given in Section \ref{Sampling diffusion bridges}. The reader is referred to the file \texttt{faberschauder.jl} in the public repository \url{https://github.com/SebaGraz/ZZDiffusionBridge/src} for the implementation of the expansion and for the generic implementation of the different variants of the Zig-Zag sampler to our package (\cite{mschauer/ZigZagBoomerang.jl}).

The first class of diffusion processes considered are diffusions with linear drift function (Subsection~\ref{linear diffusion model}). This is a special case, where our method does not require the subsampling technique described in Subsection~\ref{sub-sampling technique sub section} and only  Algorithm \ref{ZigZaglocal} has been employed. Notice that for this class, the transition kernel of the conditioned process is known. In Subsection~\ref{sin model}, we apply our method for diffusions which substantially differ from Brownian motions, being highly non-linear and multimodal and therefore creating challenging bridge distributions for standard MCMC. Here we use the the fully  local algorithm (Algorithm \ref{alg: ZigZaglDoublyLocal}).
In the specific example considered, the implementation of the Zig-Zag  sampler is facilitated by the drift function and its derivatives being bounded and therefore a bounded Poisson rate for the subsampling technique is available. In view of this, we choose for the third numerical experiment a diffusion with unbounded drift (Subsection~\ref{Diffusion model with unbounded drift}). For all the models, Assumptions \ref{A0}, \ref{A1} and \ref{A2} are immediate to verify and Assumption~\ref{Assumption dependecy structure} is satisfied. For each experiment, the burn-in $\tau_{\text{burn-in}}$ and final clock $\tau_{\text{final}}$  are manually tuned by inspecting the trace of $\xi^N$ and ensuring that the process reached stationarity before $\tau_{\text{burn-in}}$ and fully explore the state space before the final clock $\tau_{\text{final}}$. The computations are performed with a conventional laptop with a 1.8GHz intel core i7-8550U processor and 8GB DDR4 RAM. We wrote the program in Julia 1.4.2 which allows profiling and optimizing the code for high performance. The program is publicly available on GitHub at \url{https://github.com/SebaGraz/ZZDiffusionBridge} where the reader can follow the documentation to reproduce the results.
\subsection{Linear diffusions}\label{linear diffusion model}
A linear stochastic differential equation conditioned to hit a final point $v_T$ has the form
\begin{equation}
    \label{linear equation}
    \dd X_t = (\alpha +  \beta X_t) \dd t + \dd W_t, \qquad X_0 = u, X_T = v_T 
\end{equation}
for some $(\alpha,\beta) \in \mathbb{R}^2$. Assumptions \ref{A0}, \ref{A1} and \ref{A2} can be easily verified. 
In this case the energy function of the target distribution is 
\[
\psi(\xi^N) = C_1 -\ln(Z^N(X))+ \frac{\|\xi^N\|^2}{2} = C_2 + \frac{1}{2} \int_0^T \left(\beta^2\left(X_t^{\xi^N}\right)^2 + 2\alpha \beta X^{\xi^N}_t \right) \dd t + \frac{\|\xi^N\|^2}{2},
\]
for some constant $C_1,C_2$. Note that $\psi$ is a quadratic function of $\xi$, which means that the target density is still Gaussian under $\mathbb{P}^{u,v_T}_N$. It follows that
\[
\partial_{\xi_k} \psi(\xi^N) =   \int_{t \in S_k} \phi_k(t) \left( \beta^2 \left(\Bar{\Bar{\phi}} (t) u +  \Bar{\phi} (t) v_T/\sqrt{T} + \sum_{j \in N_k}\phi_j \xi_j \right) + \alpha \beta )\right) \dd t + \xi_k.
\] 
Interchanging the integral and the sum, this becomes 
\[
\partial_{\xi_k} \psi(\xi^N) =  \beta^2 \left(   \Bar{\Bar{\Phi}}_k u + \Bar{\Phi}_k v_T/\sqrt{T} + \sum_{j \in N_k} \Phi_{jk}  \xi_j \right) + \alpha \beta \Phi_k + \xi_k,
\]
 where $\Phi_k = \int \phi_k \dd t$, $\Phi_{jk} = \int \phi_k \phi_j \dd t$, $\Bar{\Phi}_k = \int \Bar{\phi} \phi_k \dd t$ and $\Bar{\Bar{\Phi}} = \int \Bar{\Bar{\phi}} \phi_k  \dd t$. This is a linear function of $\xi^N$ and, for each $i$, the event times with rates $\lambda_i$, see \eqref{Poisson rates}, can be directly simulated without upper bounds. Figure~\ref{figure ou} shows samples from the resulting diffusion bridge measure with $\alpha = -5, \beta = -1$ obtained with this method running the Zig-Zag  sampler for $\tau_{\text{final}} = 1000$, with a burn-in time of $\tau_{\text{burn-in}} = 10$. The closed form of the expansion of linear processes, or more generally, reciprocal linear processes,
 with the Faber-Schauder basis was also found and used in \textcite{van2018adaptive} for the problem of nonparametric drift estimation of diffusion processes. The results are validated by computing analytically the density of the random variable $X_{T/2}$ (which, for the linear case, is known in close form) and comparing this with its empirical density obtained from one sample of the Zig-Zag  process (see Figure~\ref{comparison}, left panel). 
\begin{figure}[ht]
    \centering
    \includegraphics[width=1\linewidth]{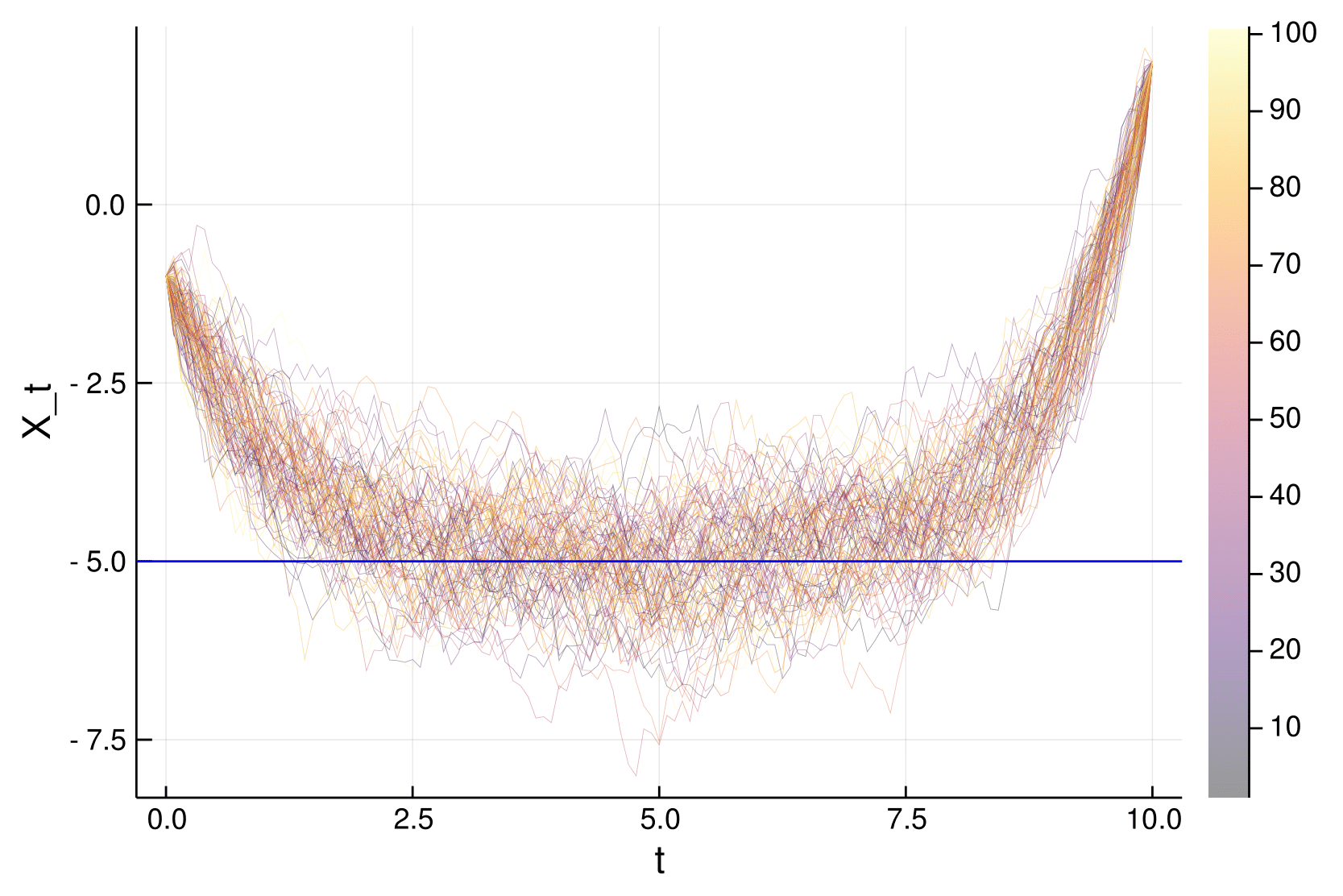}
        \caption{Simulation of the diffusion bridge measure (100 samples) given by equation (\ref{linear equation}) starting at $-1.0$ and conditioned to hit $2.0$ at $T=10$. $\alpha = -5.0, \beta = -1.0$ which is equivalent to a mean reverting process with mean reversion at $x = -5$ (straight line). The truncation level is $N = 6$, final clock $\tau_{\text{final}} = 1000$ and burn-in $\tau_{\text{burn-in}} = 10$.}
    \label{figure ou}
\end{figure}
\subsection{Non-linear multi-modal diffusions}\label{sin model}

The stochastic differential equation considered here has the form
\begin{equation}
    \label{sin diffusion model}
    \dd X_t = \alpha \sin(X_t)\dd t +  \dd W_t, \qquad X_0 = u, X_T= v_T
\end{equation}
for some $\alpha \ge 0$. When $\alpha = 0 $ the process is a standard Brownian motion while for positive $\alpha$, the process is attracted to its stable points $(2k -1)\pi, \, k \in \mathbb{N}.$ Assumption~\ref{A0},~\ref{A1},~\ref{A2} follow from drift, its primitive and its derivative being globally bounded. Fixing $N$, the energy function is given by
$$
\psi(\xi^N) =  \frac{\alpha}{2} \int_0^T \left( \alpha \sin^2 (X^{\xi^N}_t) + \cos(X^{\xi^N}_t)\right) \dd t +  \frac{\|\xi^N\|^2}{2}.
$$
Using trigonometric identities, we obtain that
$$
\partial_{\xi_k} \psi(\xi^N) =  \frac{1}{2} \int_{S_k} \phi_k(t) \left( \alpha^2 \sin\left(2X^{\xi^N, k}_t\right) - \alpha \sin\left(X^{\xi^N, k}_t \right)\right) \dd t + \xi_k
$$ 
where $ X^{\xi^N, k}_t := \Bar{\Bar{\phi}}(t) u + \Bar{\phi}(t) v_T/\sqrt{T} + \sum_{j \in N_k}\phi_j(t) \xi_j$.
To avoid the need to find the roots of equation \eqref{find the root lambda} we apply the subsampling technique described in  Subsection~\ref{sub-sampling technique sub section}.
Since the drift and its derivatives are bounded, we can easily find the following upper bound for (\ref{unbiased estimator}):
\begin{equation}
    \label{eq: bound}
    \bar{\lambda}_k(t) = |\theta_k|a_1 + (\theta_k\xi_k(t))^+,
\end{equation}
with $a_1 = \Bar \Phi_k S_k (\alpha^2 + \alpha)/2$, $\Bar \Phi_k = \max(\phi_k)$ and $\xi_k(t) = \xi_k + \theta_k t$. In this case, the upper bound $\bar \lambda_i$ is a function only of the coefficient $\xi_i$. 
Figure~\ref{fig:sinsde} shows the results obtained with this method setting $\alpha = 0.7$. For this diffusion, the non-linearity and multiple modes make the mixing of the Zig-Zag  sampler slower so we set $\tau_{\text{final}} = 10000$ and burn-in $\tau_{\text{burn-in}} = 10$.  
\begin{figure}[ht]
    \centering
    \includegraphics[width=1\linewidth]{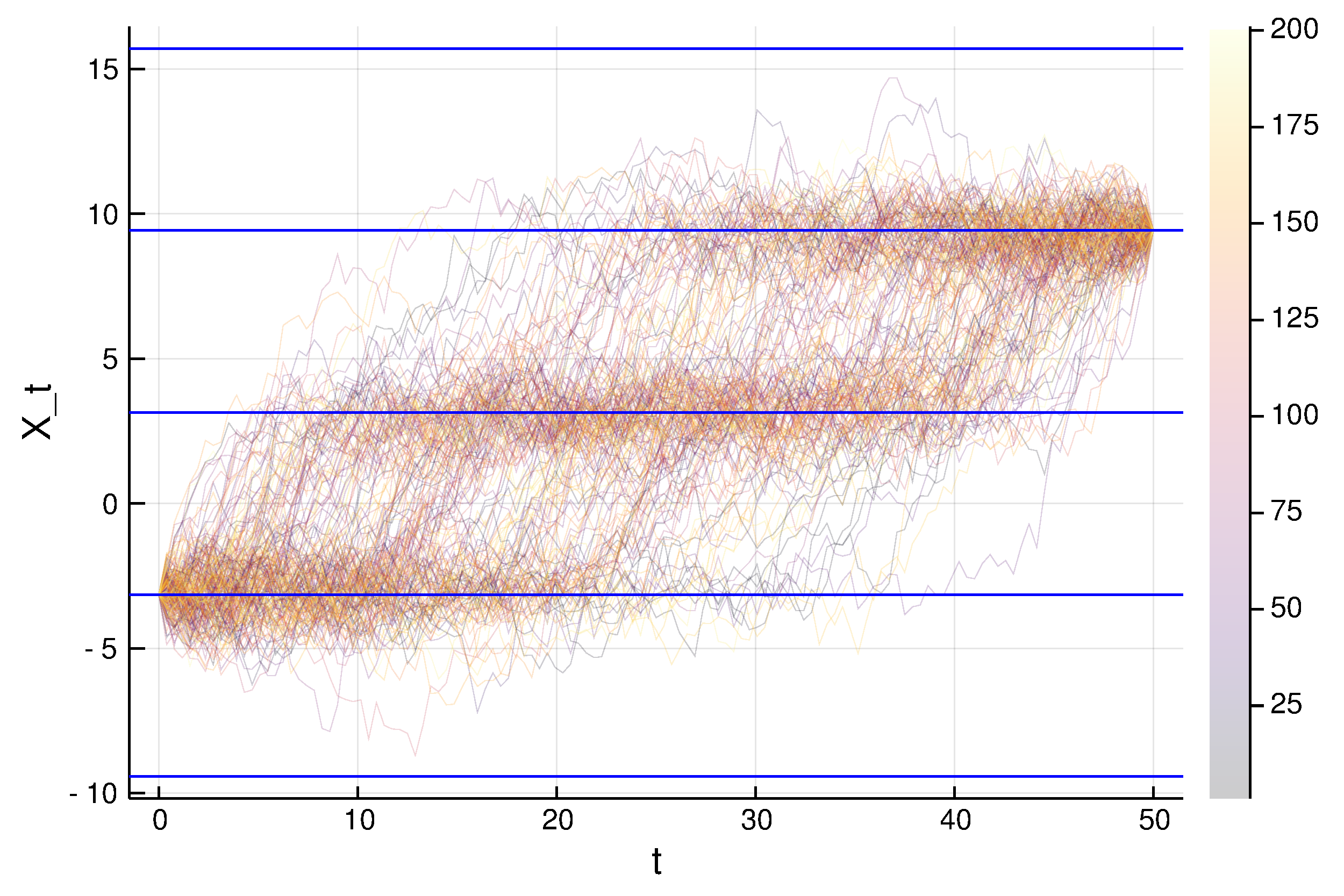}
    \caption{Simulation of the diffusion bridge measure (200 samples) given by equation (\ref{sin diffusion model}) with $\alpha = 0.7 $ starting at $-\pi$ at time $0$ and hitting $3\pi$ at $T=50$. Truncation level $N = 6$, final clock $\tau_{\text{final}} = 10000$ and burn-in $10$. The straight horizontal lines are the attraction points of the process.}
    \label{fig:sinsde}
\end{figure}

Analysing the goodness of the empirical diffusion bridge distribution obtained is a difficult task since the true conditional distribution is not known in a tractable form. We start by checking if some geometrical properties of the diffusion bridge distributions are preserved in the simulations. For example, in Figure~\ref{fig:sinsde}, it can be noticed that the diffusion is attracted to the stable points $ \pm \pi, \pm 3\pi,...$, and symmetric (geometrically speaking, after rotation) around the vertical axes $t = T/2$. We furthermore validate our method by simulating forward diffusion processes, using Euler discretization in a fine grid, and retaining only the paths which end in a $\epsilon$-ball of a certain point at time $T$ ($\epsilon$-ball forward simulation). If the final point is such that the probability of ending in this $\epsilon$-ball is high enough, we can create in this way a sample from the approximated bridge and compare it to the samples obtained from the Zig-Zag. The right panel of Figure~\ref{comparison} shows the joint empirical distribution with the two methods of the first quarter and third quarter random variables. Finally, Figure~\ref{figure qqplots} illustrates that the marginal distribution of the coefficients in higher levels is approximately Gaussian and the non-linearity of the process is absorbed by the coefficients in low levels.

\begin{figure}[ht]
\centering
\includegraphics[width=0.35\linewidth]{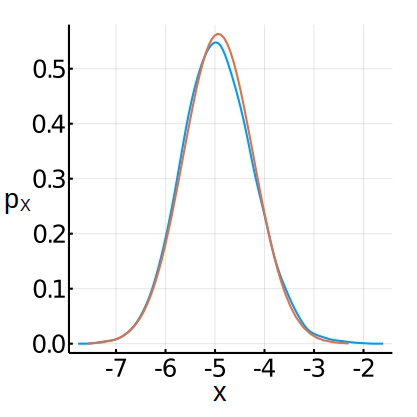}\includegraphics[width=0.6\linewidth]{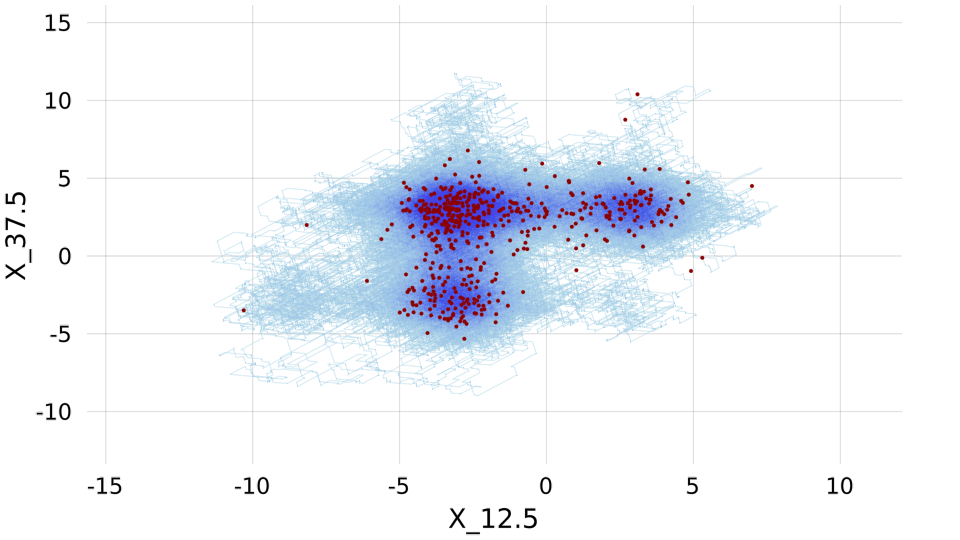}
\caption{On the left panel: comparison between empirical distribution (blue line, computed with a kernel estimator) and the exact distribution (red line) of the mid-point random variable $X_{5}$ for the linear diffusion (equation $\ref{linear equation}$) with $a = -5$ and $b = -1 $. The empirical distribution has been extracted from the same experiment shown in Figure~\ref{figure ou}. On the right panel: comparison between the joint distribution of the variables $X_{T/4}$ and $X_{3T/4}$ of the process given in equation (\ref{sin diffusion model}) starting at $-\pi$ and hitting $\pi$ at $T=50$. The scatter plot with red dots are obtained with $\epsilon$-ball Euler simulation with $\epsilon = 0.1$ and discretization $\Delta t = 0.0005 $ while the blue continuous path is the Zig-Zag path.} 
\label{comparison}

\end{figure}
 \begin{figure}[ht]
    \centering
    \includegraphics[width=1\linewidth]{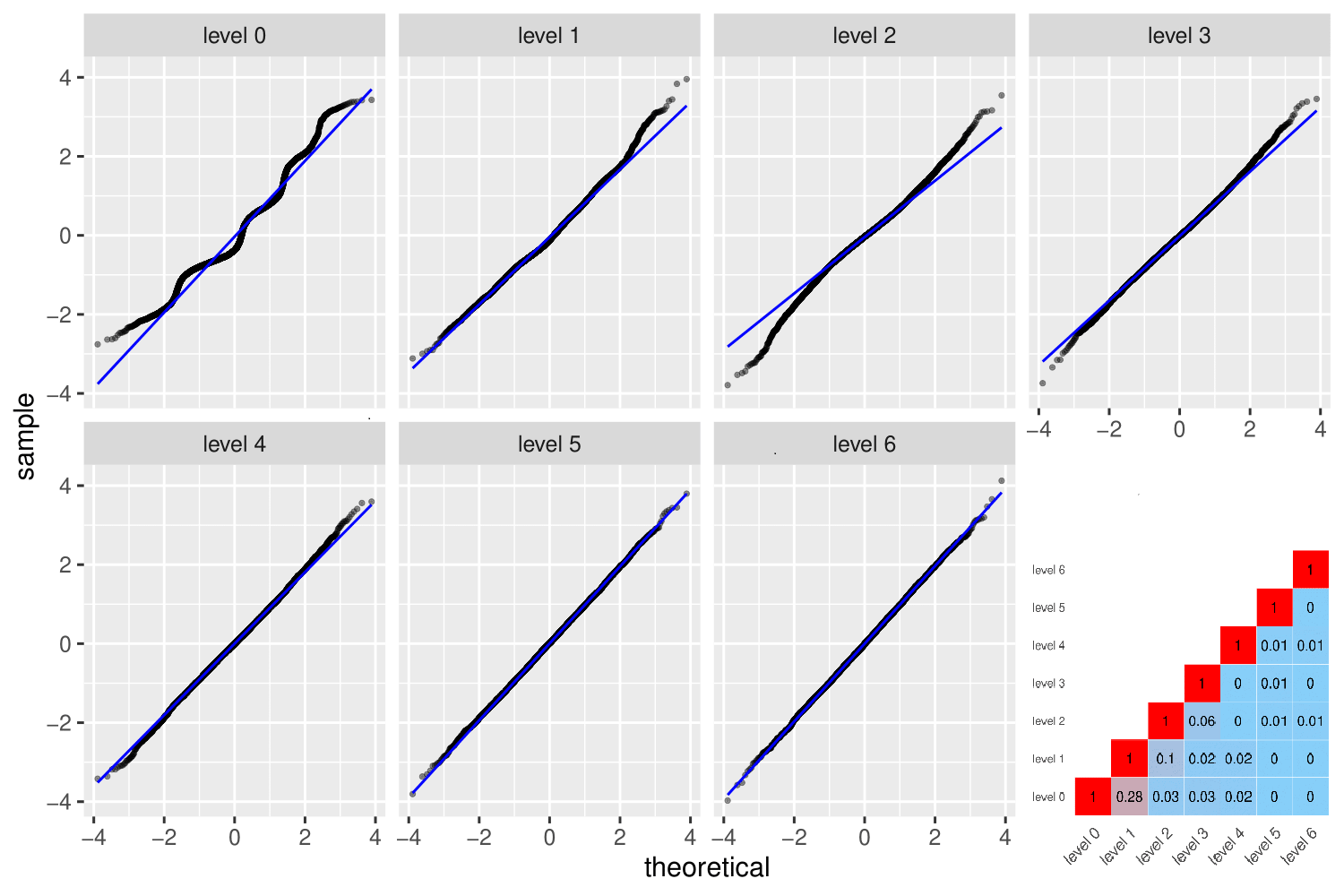}
      \caption{Q-Q (quantile-quantile) plot against standard normal distributions of the sample path of 7 coefficients respectively at level  $0, 1, 2, 3, 4, 5, 6$ targeting the conditional bridge measure given by equation (\ref{sin diffusion model}) with $\alpha = 0.7$ and initial point $u= 0$ and final point  $v= 0$ at $T = 100$. On the bottom right panel, the heatmap of the absolute value of the sample correlation between the coefficients at different levels. The blue straight lines correspond to the marginal measures of the coefficients relatively to a Brownian bridge.}
  \label{figure qqplots}
\end{figure}
 
\subsection{Diffusions  with unbounded drift}\label{Diffusion model with unbounded drift} Here we consider \textit{stochastic exponential logistic models}. For this class, the process grows exponentially with rate $r$ until it reaches its saturation point $K$. Its dynamics are perturbed by noise which grows as the population grows. The resulting stochastic differential equation takes the form 
\begin{equation}
    \label{exponential growth model}
    \dd Y_t = rY_t (1 - Y_t/K) \dd t + \beta Y_t \dd W_t, \qquad X_0 = u>0, \quad X_T = v_T>0. 
\end{equation}
We can transform the process in order to get a new process with unitary diffusivity $\sigma = 1$ (Lamperti transform with $X_t = -\log(Y_t)/\beta$). The transformed differential equation becomes
$$
\dd X_t = (c_1 + c_2 e^{-\beta X_t})\dd t + \dd W_t, \qquad X_0 = -\log(u)/\beta, \quad X_T = -\log(v)/\beta.  
$$
with $c_1 = \beta/2 - r/\beta$ and $c_2 = r/(\beta K).$ Note that the  drift function $b$ of the transformed process is not global Lipschitz continuous. Nevertheless Assumptions  \ref{A1} and \ref{A2} are satisfied and by Remark \ref{Remark for logstic}, also Assumption~\ref{A0} is verified.  
In this case, the partial derivative of the energy function is given by
$$
\partial_k \psi(\xi^N)  = \frac{1}{2}\int_{S_k} \phi_k(s) \left(  a_1 e^{-\beta X^{\xi^N}_s} -  a_2 e^{-2\beta X^{\xi^N}_s} \right)\dd s + \xi_k,
$$
where $a_1 = 2r^2/(\beta K), \, a_2 = a_1/K$. As before, it is not possible to simulate directly the first event time using the Poisson rates given by equation (\ref{Poisson rates}). 
The subsampling technique requires an upper bound for the unbiased estimator (\ref{unbiased estimator}). Define the following quantities
$$
b^{(1)}_k := \inf_{s \in S_k} \left\{\Bar{\Bar{\phi}} (s) u_0 +  \Bar{\phi} (s) v_T/\sqrt{T} + \sum_{i \in N_k}  \phi_i(s) \xi_i\right\}, \qquad b^{(2)}_k := \inf_{s \in S_k} \left\{\sum_{i \in N_k}\phi_i(s) \theta_i\right\}.
$$
For any $a,b,c \in \mathbb{R}, \, (a + b + c)^+ \le (a)^+ +  (b)^+ +  (c)^+$ and hence a valid upper bound for the Poisson rate (\ref{unbiased estimator}) is given by
    \begin{equation}
        \label{eq: bound exponential logistc}
        \Bar{\lambda}_k(t) = \lambda_k^{(1)}(t) + \lambda_k^{(2)}(t) + \lambda_k^{(3)}(t)
    \end{equation}
    with 
    \[
     \lambda_k^{(1)}(t) = \max\left(0, \theta_k \xi_k(t)\right),
    \]
    \[
    \lambda_k^{(2)}(t) = \max\left(0, \frac{1}{2} \theta_k \Bar{\phi}_k S_k z^{(1)}_k e^{-\beta^\star_k \, t}\right),
    \]
    \[
    \lambda_k^{(3)}(t) = \max\left(0, -\frac{1}{2} \theta_k \Bar{\phi}_k S_k z^{(2)}_k e^{2\beta^\star_k \, t}\right)
    \]
and $$z^{(1)}_k = a_1 \exp(-\beta b^{(1)}_k),\, z^{(2)}_k = z^{(1)}_k \exp(-\beta b^{(1)}_k),\, \beta^\star_k = -\beta b^{(2)}_k, \, \Bar{\phi_i} = \max_s \phi_i(s).$$ 

Using the \textit{superposition theorem} (see for example \cite{grimmett2001probability}), we can simulate a waiting time with Poisson rate (\ref{eq: bound exponential logistc}) by means of simulating three waiting times according to the Poisson rates $\lambda_k^{(1)}, \lambda_k^{(2)}, \lambda_k^{(3)}$ and then take the minimum of the three realizations. Since at any time $t>0$, either $\lambda_k^{(2)}(t)$ or $\lambda_k^{(3)}(t)$ is 0, we just need to evaluate two waiting times. Figure~\ref{fig:sub2} shows the results obtained with our method for this process. The final clock of the Zig-Zag sampler is set to $T^{\star} = 1000$ and initial burn-in time $\tau_{\text{burn-in}} = 10$.
\begin{figure}[ht]
    \centering
    \includegraphics[width=1\linewidth]{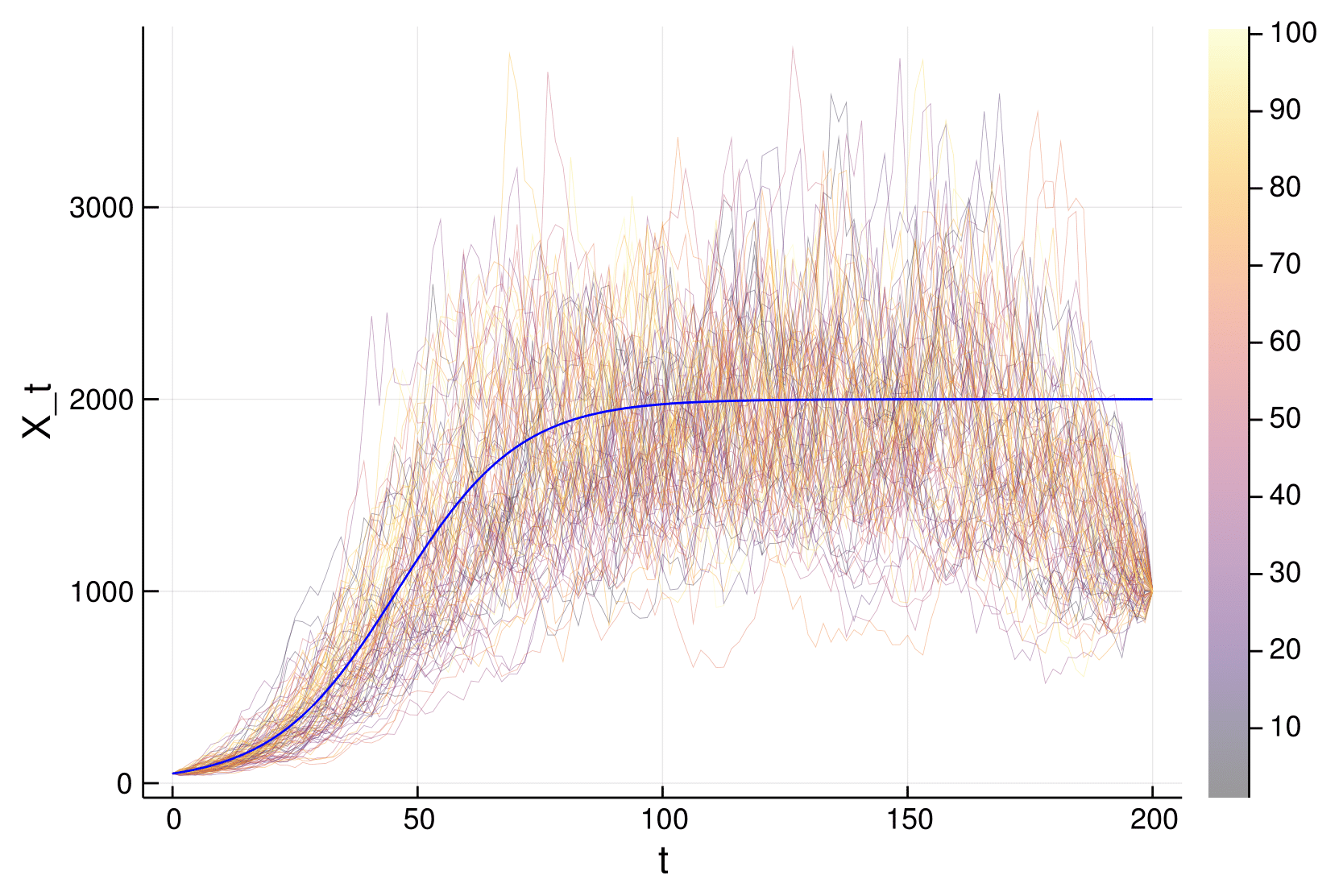}
    \caption{Simulation of the diffusion bridge measure (100 samples) given by the logistic growth model (equation (\ref{exponential growth model})) with parameters $K= 2000, r= 0.08,\beta = 0.1$, starting at the value $50$ and hitting $1000$ at time $200$. Truncation level $N = 6$, final clock $\tau_{\text{final}} = 1000$ and burn-in $\tau_{\text{burn-in}}=10$ . The blue smooth line is the solution of the deterministic logistic model without final condition.}
    \label{fig:sub2}
\end{figure}

\subsection{Numerical comparisons}\label{Subsection: Numerical comparisons}
 In this section we benchmark the fully local Zig-Zag sampler against the Metropolis-adjusted Langevin algorithm (MALA) (\cite{10.2307/3318418}),  Hamiltonian Monte Carlo (HMC) (\cite{DUANE1987216}) and another well known PDMP, the Bouncy particle sampler  (\cite{2015arXiv151002451B}). The Bouncy Particle sampler can use the exact subsampling technique in a very similar way as explained in Subsection~\ref{sub-sampling technique sub section}. According to the scaling limit results obtained in \textcite{bierkens2018highdimensional}, the Zig-Zag is more efficient compared to the Bouncy Particle sampler in a high dimensional setting when the  conditional dependency graph corresponding to the target measure exhibits sparsity (which clearly is the case here). The MALA sampler is a well known discrete-time Markov chain Monte Carlo method which performs informed updates through the gradient of the target distribution. HMC  is considered a state-of-the-art algorithm.  In contrast to PDMPs,
for  HMC and MALA the gradient needs to be fully evaluated and no subsampling methods can be exploited. Thus,  the integral in \eqref{partial} needs to be computed numerically, introducing bias. Furthermore, contrary to PDMPs, the resulting Markov chain is reversible. We study the performance of the samplers for the stochastic differential equation  \eqref{sin diffusion model} with $u, v = 0$ and the time horizon $T = 100$ and we let  $\alpha$ vary. As $\alpha$ increases, the target distribution on the coefficients presents higher peaks and valleys and is therefore a challenging distribution for general Markov chain Monte Carlo methods. We fix the refreshment rate of the Bouncy Particle sampler to 1 to avoid a degenerate behaviour and implement the MALA algorithm with  adaptive step size over 250,000 iterations. 
We used the automatically tuned  dynamic integration time HMC Algorithm (\cite{betancourt2018conceptual}) with $3,000$ iterations and with diagonal mass matrix and integrator step size both adaptively tuned in a warm-up phase of $2,000$ iterations, with the latter adapted using a dual-averaging algorithm (\cite{hoffman2014no}) with target acceptance statistic of 0.8. The algorithm is provided in the package \texttt{AdvancedHMC.jl} (\cite{ge2018t}). 
The integral appearing in the gradient of the energy function is computed for the MALA sampler and for the HMC sampler numerically with a simple Euler integration scheme over $2^{N+1}$ points, where $N$ is the truncation level which is fixed to 6 for all the experiments. The final clock for the PDMPs is $T' = 25,000$. We also include the numerical results of two variants of the Zig-Zag sampler: 
\begin{itemize}
    \item[(ZZv1)] where  the partial derivative in \eqref{partial} is estimated by averaging over multiple independent realizations of \eqref{eq: unbiased Poisson estimator}, with the number of realizations is proportional to the length of the range of the integral in \eqref{partial};
\item[(ZZv2)]  where   the partial derivative in \eqref{partial} is estimated  by decomposing the range of the integral into $N$ subintervals (with $N$ proportional to the length of the range of the integral) and evaluating the integrand at a random point drawn inside each subinterval.
\end{itemize}
 These variants of the Zig-Zag have been proposed after noticing that the coefficients at low levels are the ones deviating the most from normality and the partial derivative with respect to those coefficients have larger support. This suggests that refining the estimates of the partial derivative of the energy function only with respect to those coefficients can be beneficial and improve the performances of the PDMPs. Figure~\ref{fig: comparison} shows the results obtained. The fully  local Zig-Zag and its variants always outperform the Bouncy Particle sampler, the MALA and the HMC with respect to the statistics considered, namely the mean, median and minimum of the effective sample size computed for each coefficient of the Faber-Schauder expansion and the effective sample size of the coefficient $\xi_{0,0}$, which gives the middle point $X_{T/2}$ and, as shown in Figure~\ref{fig: comparison}, is one of the most difficult coefficients to sample.
\begin{figure}[ht!]
    \centering
    \includegraphics[width=1.0\linewidth]{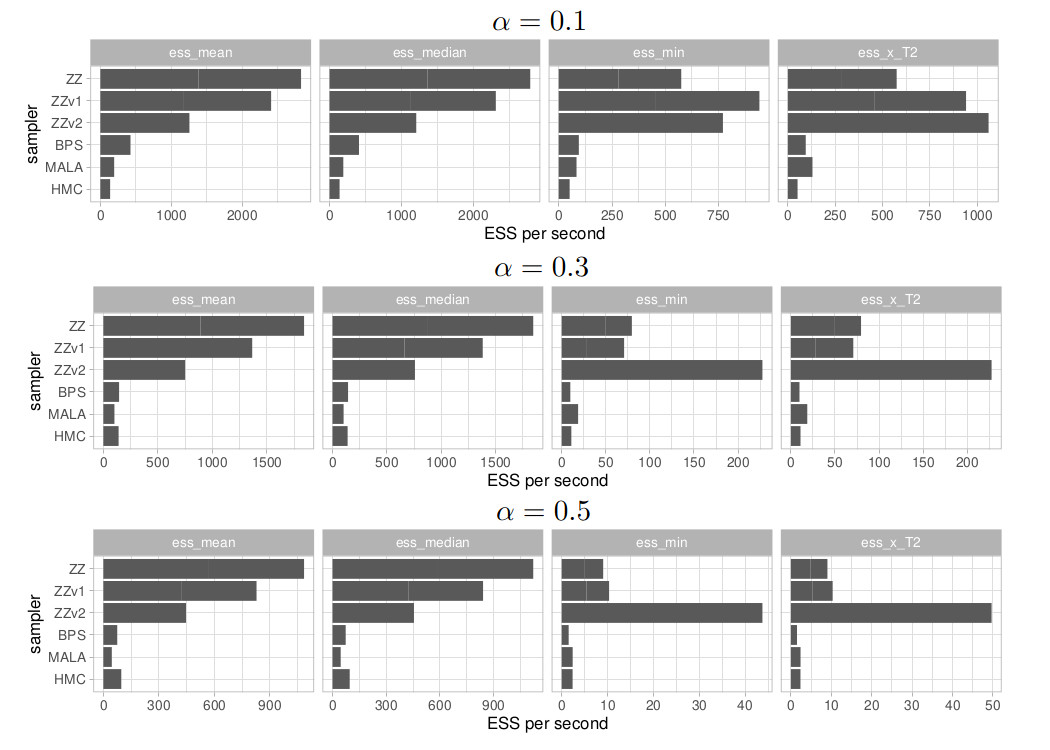}
    \caption{Performance comparison of the fully  local Zig-ZaZ (ZZ), its variants (ZZv1 and ZZv2), the Bouncy Particle sampler with Subsampling (refreshment rate set to $1$), MALA and HMC sampler. The performance measure considered here are respectively the effective sample size (ESS) of the middle point $X_{T/2}$, the median and the minimum  of the ESS over the dimension of the coefficients of the expansion. The target diffusion bridge with drift $b(x) = \alpha \sin(x)$ with $u, v = 0$ and $T = 50$ and truncation level $N = 6$. The final clock for the PDMPs is set to $T' = 25000$, the number of iterations for the MALA is set to be 250000 with adaptive time step targeting the acceptance rate 0.6 (\cite{roberts1998optimal}) and the number of iteration for the HMC is $3000$ with the algorithm fine-tuned by the package \texttt{AdvancedHMC.jl}. All the quantities are normalized by the runtime of execution. The asymptotic variance estimate used for computing the ESS is obtained using batch means.  Notice that, while the subsampling technique adopted for the piecewise deterministic Monte Carlo methods does not introduce bias on the target distribution, the numerical integration adopted for the MALA and HMC samplers introduces bias on the target distribution.}
    \label{fig: comparison}
\end{figure}

%% file: part7.tex
\section{Extensions}\label{Subsection: Informal scaling analysis}
 In this section we briefly sketch the extension of the approach presented in Section~\ref{faber schauder expansion of diffusion processes}  to a class of multi-dimensional diffusion bridges. Then we study the scaling properties of the algorithm with respect to three quantities: the time horizon of the diffusion bridge $T$, the truncation level $N$ and the dimensionality of the diffusion bridge $d$.
\subsection{Multivariate diffusion bridge}\label{subsec: multivariate diffusion bridge}
Consider a $d$-dimensional diffusion bridge given the stochastic differential equation
\begin{equation*}
    \dd X_t = \nabla B(X_t) \dd t + \dd W_t, \quad  X_0 = u,X_T = v_T, \quad u,v_T \in \RR^d,
\end{equation*}
where $(W_t)_{t\ge 0}$ is a $d$-dimensional Wiener process and $\nabla B: \RR^d \mapsto \RR^d$ is a conservative vector field, i.e. the gradient of some scalar-valued function $B$. Denote its law by $\PP^{u,v_T}$. Similarly to equation \eqref{girsanov Z_2}, under mild assumption on $\nabla B(X_t)$, we can write the change of measure between $\PP^{u,v_T}$ and the standard $d$ dimensional Wiener bridge measure $\QQ^{u,v_T}$ as
\begin{equation*}
    \frac{\dd \PP^{u,v_T}}{\dd \QQ^{u,v_T}}(X) = C\exp\left\{B(X_T) - B(X_0) - \frac 1 2 \int_0^T  \|b(X_t)\|^2 + \Delta B(X_t)  \, \dd t \right\},
\end{equation*}
where $b = \nabla B$, $\Delta B$ is the Laplacian of $B$ and $C$ is a normalization constant which depends on $u,v_T$ and $T$. It is straightforward to derive an equivalent approximated measure as done in equation \eqref{approx diffusion bridge } and prove Theorem~\ref{main theorem} in the multi-dimensional setting. In this case the $d$ dimensional diffusion bridge measure is approximated by the same truncated expansion of equation \eqref{sum expansion} with coefficients $\xi_{i,j}, i = 0,...,N; \, j = 0,...,2^N$ which now are $d$  dimensional random vectors. The total dimensionality of the target density for diffusion bridges becomes $d(2^{N+1} -1)$ . Similarly to the one dimensional case, Proposition~\ref{conditional independece structure} holds (the proof follows in a similar fashion of the proof of Proposition~\ref{conditional independece structure} and is omitted for brevity). The Poisson rates  $\lambda^k_{i,j}$ (where, $k \in \{1,...,d\}$  defines the coordinate of the $d$ dimensional process) are functions of the sets $N_{i,j}^k$ which have maximum admissible size  $|N_{i,j}^k| = d(2^{N-i+1} + i -1) \le d(2^{N+1} - 1)$ so that Assumption~\ref{Assumption dependecy structure} holds.

\subsection{Scaling for large \texorpdfstring{$T, N, d$}{T, N, d}}\label{subsection: sca;omg for large T, N, d}
The following scaling analysis serves as preliminary work for future explorations. The expected run time of the fully local Zig-Zag sampler (Algorithm \ref{alg: ZigZaglDoublyLocal}) is intimately related with the number of Poisson event times for a fixed final clock $\tau_{\text{final}}$ and the conditional independence structure appearing in the target measure. The former is determined by the size of the Poisson bounding rates $\bar \lambda_1,...,\bar \lambda_{M}$ while the latter is defined by the sets $N_{1},...,N_{M}$ and determine the complexity of the \textit{local step} of Algorithm \ref{ZigZaglocal}. 

\begin{rmk}\label{prop: lambda in terms of T and i}
    For a fixed position and velocity, the Poisson bounding rates used in the Zig-Zag sampler with subsampling (Algorithm \ref{ZigZagsubsampler}) for diffusion bridges are of the form $\bar\lambda_{i,j} = C_1  T^{3/2}2^{-3i/2} + C_2, \, i = 0, 1,...,N; \, j = 0,1,...,2^i-1$, for some terms $C_1$ and $C_2$ which do not depend on $i$ and $T$.
\end{rmk}
\begin{proof}
For every $i = 0, 1,...,N; \, j = 0,1,...,2^i-1$, the time horizon $T$ and scaling index $i$ enter in the bounding rates of \eqref{eq: bounding Poisson rate} through the terms $S_{i,j}$ and $\bar \phi_{i,j}$. The first term is of $\mathcal{O}(T 2^{-i})$ and the second one is of $\mathcal{O}(\sqrt{T}2^{-i/2})$. 
\end{proof}

Proposition~\ref{prop: lambda in terms of T and i} helps understanding how the complexity of the algorithm scales as $T$ grows and as the truncation level $N$ grows. As $T$ grows, the Poisson rates increase with order $T^{3/2}$ so that the total number of Poisson events for a fixed Zig-Zag clock increases with the same order. 

Furthermore, as the truncation level $N$ grows, the change of measure  affects less and less the coefficients in high levels and the partial derivative of the energy function goes to zero with rate $2^{-3N/2}$) implying that the for large $N$, $\Bar \lambda_{N,j} \approx C_2 = (\xi_{N,j} \theta_{N,j})^+$ (which is the Poisson rate for the Brownian bridge). As a consequence, the Poisson processes of the coefficients in high levels ($i$ large) will be approximately independent with all the other coefficients and not function of the level $i$ so that the complexity of Algorithm~\ref{alg: ZigZaglDoublyLocal} scales approximately linearly with the number of mesh points. This is opposed to the standard Zig-Zag algorithm (Algorithm~\ref{ZigZag1}) which does not take advantage of the approximate independence of the coefficients in high levels so that the $2^{N+1}-1$ waiting times have to be renovated at every reflection of each coefficient.

The scaling result under mesh refinement (when $N$ grows) is unsatisfactory as the algorithm deteriorates when the resolution of the path increases. A partial solution can be obtained by letting the absolute value of the marginal velocities $|\theta_{N,j}|$ decrease as $N$ increases. This would enhance the scaling property of the algorithm under mesh refinement at the cost of a slow mixing of high level components. An alternative solution is considered in  \textcite{bierkens2020boomerang} where the authors enhance the scaling property of the algorithm by replacing the Zig-Zag algorithm with the \textit{Factorised Boomerang sampler}. The Factorised Boomerang sampler differs form the Zig-Zag by having curved trajectories which are invariant to a prescribed Gaussian measure. This allows the process to sample from the Gaussian measure (Brownian bridge measure) at barely no cost. However, the main drawback of the factorized Boomerang sampler is the current limiting techniques for simulating Poisson times given the curved trajectories which lead to Poisson upper bounds which are not tight.  

 Finally, when the dimensionality of the diffusion bridge is $d \gg 1$, both the dimensionality of the target density of the Zig-Zag sampler and the sets $N_{i,j}^k$ for $i = 0,...,N; \, j = 0,...,2^i-1;\, k = 1,...,d$ grow linearly with $d$ so that, in general, we expect the computational time to grow with rate $d^2$. When the drift of the multidimensional bridge presents a sparse structure, i.e. not all coordinates of the differential equation interact directly with each other, as common in the high dimensional case arising from discretized stochastic partial differential equations (e.g. \cite{mider2017continuousdiscrete}, Section 6), the size of those sets reduces considerably until the extreme case of $d$ independent diffusion bridges where the sets $N_{i,j}^k$ are not anymore a function of $d$ and clearly the complexity grows linearly with the dimensionality $d$.

%% file: part8.tex
\section{Conclusions}\label{Conclusions}
In this paper we have introduced a new method for the simulation of diffusion bridges which substantially differs from existing methods by using the Zig-Zag  sampler and the basis of representation adopted. We motivated both choices and presented  the method and its implementation. The resulting simulated bridge measures are shown to be close to the real measures, even for low dimensional approximations and bridges which are highly non-linear. We took advantage of the subsampling technique and a local version of the Zig-Zag  to sample high dimensional approximation to conditional measures of diffusions with intractable transition densities. The subsampling technique is a key property in favour of using piecewise deterministic Monte Carlo methods for diffusion bridges (and whenever the target measure is expressed as an integral that requires numerical evaluation). However, the main limitation found for these methods is that they rely on upper bounds of the Poisson rates which are model-specific. Upper bounds for PDMC are easily derived in situations where the log-likelihood has a bounded Hessian. In our setting this means that we wish for the function $b^2(x) - b'(x)$ to have bounded second derivative. In other cases, tailor made bounds need to be derived which can be substantially more complicated. Furthermore, the performance of these samplers can be affected if the upper bounds are too large.

In conclusion, this is the first time (to our knowledge) the Zig-Zag  has been employed in a high dimensional practical setting. We claim that the promising results will open research toward applications of the Zig-Zag  for high dimensional problems. We mention below some possible extensions of the methodology proposed which are left for future research:
\begin{enumerate}[label=(\alph*)]
    \item The hierarchical structure of the Faber-Schauder basis suggests that the Zig-Zag  should explore the space at different velocities to achieve optimal performance. Unfortunately, it is not immediately clear how to tune the velocity vector;
    \item In Section~\ref{Subsection: Informal scaling analysis} we anticipated the possibility to simulate multidimensional diffusion bridges. In order to generalize the results presented in this paper, we assumed the drift being a conservative vector field. 
    In order to relax this limiting assumption, new convergence results have to be derived which deal explicitly with the stochastic integral appearing in equation \eqref{girsanov Z_1}.
    \item The driving motivation for proposing this methodology is to perform parameter estimation of a discretely observed diffusion model. For this purpose, the Zig-Zag  sampler runs jointly on the augmented path space given by the coefficients $\xi$ and the parameter space $\Theta$.
\end{enumerate}
\section*{Acknowledgement}
This work is part of the research programme \textit{Bayesian inference for high dimensional processes} with project number 613.009.034c, which is (partly) financed by the Dutch Research Council (NWO) under the  \textit{Stochastics -- Theoretical and Applied Research} (STAR) grant. J.~Bierkens acknowledges support by the NWO for the research project \textit{Zig-zagging through computational barriers} with project number 016.Vidi.189.043. The authors are thankful to Gareth Roberts and Marcin Mider for fruitful discussions and grateful to the reviewers for valuable input.

%% file: appendix.tex
\section{Factorization of the diffusion bridge measure}\label{App: fact diff bridge}
Here we derive rigorously the conditional independence structure of the coefficients which arise from the compact support of the Faber-Schauder functions as shown in Figure~\ref{Fig: support of the fs functions}. Recall that the relation $ \xi_{i,j} \ll \xi_{k,l}$ holds if $S_{k,l} \subset S_{i,j}$ and in that case we refer to $\xi_{i,j}$ as the \emph{ancestor} of $\xi_{k,l}$ (and conversely $\xi_{k,l}$ as the \emph{descendant}). Notice that each coefficient is both descendant and ancestor of itself. 
\begin{prop} 
\label{conditional independece structure}
\begin{sloppypar}
(Conditional independence structure) Denote the set of common ancestors of $\xi_{i,j}$ and $\xi_{k,l}$ by $A_{(i,j; k,l)} := \{ \xi_{h,d} \colon  \xi_{h,d} \ll \xi_{k,l} \wedge  \xi_{h,d} \ll \xi_{i,j} \}$. Under $\mathbb{P}^{v_T}_N$, $\xi_{i,j}$ is conditionally independent from $\xi_{k,l}$, given the set $A_{(i,j;k,l)}$, whenever the interior of the supports of their basis function are disjoint that is neither $\xi_{i,j} \ll \xi_{k,l}$ nor $\xi_{k,l} \ll \xi_{i,j}$ is satisfied. 
\end{sloppypar}
\end{prop}

\begin{proof}
  For $i= 1,...,N; \, j = 1,...,2^i - 1$, define the vectors of ancestors and descendants of $\xi_{i,j}$ as $\xi^{(i,j)} := \{\xi_{m,n} \colon \xi_{m,n} \ll \xi_{i,j} \vee \xi_{m,n} \gg \xi_{i,j}\}$.
  Assume, without loss of generality, that $i \le k$ and consider two coefficients $\xi_{i,j}, \xi_{k,l}$. We factorize $Z^N(X)$ by partitioning the integration interval $[0,T]$ in a sequence of sub-intervals $S_{k,0}, S_{k,1},...,S_{k,2^k -1}$ so that
\begin{equation}
    \label{factorization}
    Z^N(X) = \prod_{p = 1}^{2^k - 1} f_{k,p}(\xi^{(k,p)}).
\end{equation}
Here \[f_{k,p}(\xi^{(k,p)}) = \exp\left(B( X^N_{\max{S_{k,p}}}) - B( X^N_{\min {S_{k,p}}}) - \frac{1}{2} \int_{S_{k,p}} b^2\left(X^{N;k,p}_s \right) + b'\left(X^{N; k,p}_s \right) \dd s \right).\]
with \[X^{N;k,p}_s = \Bar{\Bar{\phi}}(s) u +  \Bar{\phi}(s)v_T/\sqrt{T} + \sum_{(i,j) \colon \xi_{i,j} \ll \xi_{k,p} } \phi_{i,j}(s) \xi_{i,j}\] and we used that $X^N_s = X^{N;i,j}_s $ when $s \in S_{i,j},\, X^N_T = \Bar{\phi}(T) v_T/\sqrt{T}$ and $X^N_0 = \Bar{\Bar{\phi}}(0) u $. Now just notice that, under this factorization, the only factor which is a function of $\xi_{k,l}$ is $f_{k,l}(\xi^{(k,l)})$. Here, if $\xi_{i,j} \not\ll \xi_{k,l}$ then $\xi^{(k,l)}$ does not contain $\xi_{i,j}$. Conversely, the factors containing $\xi_{i,j}$ are those $f_{k,p}(\xi^{(k,p)})$ such that $ \xi_{i,j} \ll \xi_{k,p} $ with $p = 0,1,...,2^k-1$. If $\xi_{i,j} \not\ll \xi_{k,l}$, none of the vectors $\xi^{(k,p)}$ contains $\xi_{k,l}$. Since, under the measure $\mathbb{Q}^{u,v_T}$, the random variables in the vector $\xi^N$ are pairwise independent, the factorization on $Z^N(X)$ defines the dependency structure of the vector $\xi^N$ under $\mathbb{P}^{v_T}_N$ so that $\xi_{i,j}$ and  $\xi_{k,l}$ are independent conditionally on their common coefficients given by the set $ A_{(i,j;k,l)}$.
\end{proof}

More intuitively, the factorization of $Z(X)$ gives rise to the dependency graph displayed in Figure~\ref{depepndence graph} which shows that the coefficients in high \textit{levels} ($i$ large) are coupled with just few other coefficients and conditionally independent from all the remaining. 
The conditional independence of the coefficients implies that the partial derivatives of the energy function (and consequently the Poisson rates given by equation \eqref{Poisson rates}) are functions of only few coefficients in the sense of Assumption~\ref{Assumption dependecy structure}. In particular the sets in Assumption~\ref{Assumption dependecy structure} (using double indexing) can be chosen as $N_{i,j} = \{ \xi_{h,d} \colon \xi_{h,d} \ll \xi_{i,j} \vee  \xi_{h,d} \gg \xi_{i,j}\}$ with size $|N_{i,j}|= 2^{N-i + 1} + i -1 $, where $N$ is the truncation level.
\input{fig4}

%% file: fig4.tex
\tikzset{main node/.style={circle,fill=blue!10, draw, minimum size=0.7cm,inner sep=0pt},
            }
            
\begin{figure}
    \centering

\begin{tikzpicture}

\node[main node] (1) {$\xi_{0,0}$};
\node[above,font=\large\bfseries] (0) [above = 1.5cm of 1] {Dependency structure};
\node[main node] (2) [left = 1.3cm of 1]  {$\xi_{1,0}$};
\node[main node] (3) [right = 1.3cm of 1] {$\xi_{1,1}$};
\node[main node] (4) [above left = 1.0cm and 0.5cm of 2] {$\xi_{2,0}$};
\node[main node] (5) [below left = 1.0cm and 0.5cm of 2] {$\xi_{2,1}$};
\node[main node] (8) [left = 1.5cm of 4] {$\xi_{3,0}$};
\node[main node] (9) [below left = 0.2cm and 1.0cm of 4] {$\xi_{3,1}$};
\node[main node] (10) [above left = 0.2cm and 1.0cm of 5] {$\xi_{3,2}$};
\node[main node] (11) [left = 1.5cm of 5] {$\xi_{3,3}$};

\node[main node] (6) [above right = 1.0cm and 0.5cm of 3] {$\xi_{2,2}$};
\node[main node] (7) [below right = 1.0cm and 0.5cm of 3] {$\xi_{2,3}$};
\node[main node] (12) [right = 1.5cm of 6] {$\xi_{3,4}$};
\node[main node] (13) [below right = 0.2cm and 1.0cm of 6] {$\xi_{3,5}$};
\node[main node] (14) [above right = 0.2cm and 1.0cm of 7] {$\xi_{3,6}$};
\node[main node] (15) [right = 1.5cm of 7] {$\xi_{3,7}$};

\path[draw]
(1) edge node {} (2)
(1) edge node {} (3)
(2) edge node {} (4)
(2) edge node {} (5)
(1) edge node {} (4)
(1) edge node {} (5)
(1) edge node {} (8)
(1) edge node {} (9)
(1) edge node {} (10)
(1) edge node {} (11)
(4) edge node {} (8)
(4) edge node {} (9)
(5) edge node {} (10)
(5) edge node {} (11)
(2) edge node {} (8)
(2) edge node {} (9)
(2) edge node {} (10)
(2) edge node {} (11)

(3) edge node {} (6)
(3) edge node {} (7)
(1) edge node {} (6)
(1) edge node {} (7)
(1) edge node {} (12)
(1) edge node {} (13)
(1) edge node {} (14)
(1) edge node {} (15)
(6) edge node {} (12)
(6) edge node {} (13)
(7) edge node {} (14)
(7) edge node {} (15)
(3) edge node {} (12)
(3) edge node {} (13)
(3) edge node {} (14)
(3) edge node {} (15)
;

\end{tikzpicture}
    \caption{ Graphical representation of the dependency structure of the random vector of the coefficients under $\mathbb{P}^{u, v_T}_N$.  $\xi_{i,j} \independent \xi_{k,l}$  conditionally on the vertices which have a direct hedge to both $\xi_{i,j}$ and $\xi_{k,l}$ if  $\xi_{i,j}$ does not have a direct edge to $\xi_{k,l}$. The dependency graph is a \textit{chordal graph}.}
    \label{depepndence graph}
\end{figure}
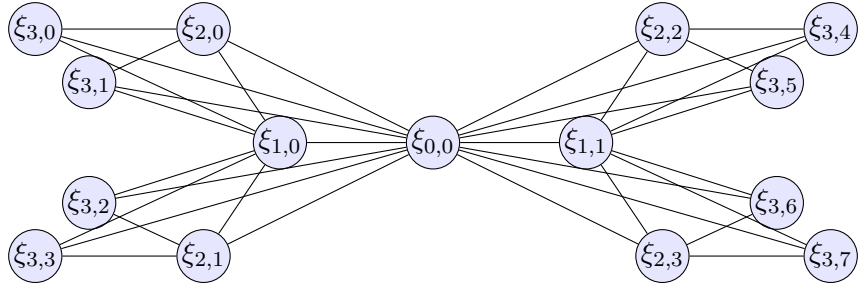

%% file: bib.bib
@article{hoffman2014no,
  title={The No-U-Turn sampler: adaptively setting path lengths in Hamiltonian Monte Carlo.},
  author={Hoffman, Matthew D and Gelman, Andrew},
  journal={J. Mach. Learn. Res.},
  volume={15},
  number={1},
  pages={1593--1623},
  year={2014}
}

@misc{betancourt2018conceptual,
      title={A Conceptual Introduction to Hamiltonian Monte Carlo}, 
      author={Michael Betancourt},
      year={2018},
      eprint={1701.02434},
      archivePrefix={arXiv},
      primaryClass={stat.ME}
}

@article{Monmarch__2020,
   title={Velocity jump processes: An alternative to multi-timestep methods for faster and accurate molecular dynamics simulations},
   volume={153},
   ISSN={1089-7690},
   url={http://dx.doi.org/10.1063/5.0005060},
   DOI={10.1063/5.0005060},
   number={2},
   journal={The Journal of Chemical Physics},
   publisher={AIP Publishing},
   author={Monmarché, Pierre and Weisman, Jérémy and Lagardère, Louis and Piquemal, Jean-Philip},
   year={2020},
   month={7},
   pages={024101}
}

@article{Michel_2019,
   title={Clock Monte Carlo methods},
   volume={99},
   ISSN={2470-0053},
   url={http://dx.doi.org/10.1103/PhysRevE.99.010105},
   DOI={10.1103/physreve.99.010105},
   number={1},
   journal={Physical Review E},
   publisher={American Physical Society (APS)},
   author={Michel, Manon and Tan, Xiaojun and Deng, Youjin},
   year={2019},
   month=1
}

@article{Faulkner_2018,
   title={All-atom computations with irreversible {M}arkov chains},
   volume={149},
   ISSN={1089-7690},
   url={http://dx.doi.org/10.1063/1.5036638},
   DOI={10.1063/1.5036638},
   number={6},
   journal={The Journal of Chemical Physics},
   publisher={AIP Publishing},
   author={Faulkner, Michael F. and Qin, Liang and Maggs, A. C. and Krauth, Werner},
   year={2018},
   month={8},
   pages={064113}
}

@article{Peters_2012,
   title={Rejection-free Monte Carlo sampling for general potentials},
   volume={85},
   ISSN={1550-2376},
   url={http://dx.doi.org/10.1103/PhysRevE.85.026703},
   DOI={10.1103/physreve.85.026703},
   number={2},
   journal={Physical Review E},
   publisher={American Physical Society (APS)},
   author={Peters, E. A. J. F. and de With, G.},
   year={2012},
   month={2}
}

@article{roberts1998optimal,
  title={Optimal scaling of discrete approximations to Langevin diffusions},
  author={Roberts, Gareth O and Rosenthal, Jeffrey S},
  journal={Journal of the Royal Statistical Society: Series B (Statistical Methodology)},
  volume={60},
  number={1},
  pages={255--268},
  year={1998},
  publisher={Wiley Online Library}
}

@misc{mschauer/ZigZagBoomerang.jl,
  doi = {10.5281/zenodo.3931118},
  note={\url{https://www.github.com/mschauer/ZigZagBoomerang.jl}},
  author = {Schauer,  Moritz and Grazzi,  Sebastiano},
  title = {{ZigZagBoomerang}: v0.5.3},
  publisher = {Zenodo},
  year = {2020},
  copyright = {Open Access}
}

@article{10.2307/3318418,
 ISSN = {13507265},
 URL = {http://www.jstor.org/stable/3318418},
 author = {Gareth O Roberts and Richard L Tweedie},
 journal = {Bernoulli},
 number = {4},
 pages = {341--363},
 publisher = {International Statistical Institute (ISI) and Bernoulli Society for Mathematical Statistics and Probability},
 title = {Exponential Convergence of Langevin Distributions and Their Discrete Approximations},
 volume = {2},
 year = {1996}
}

@misc{bierkens2020boomerang,
    title={The Boomerang Sampler},
    author={Joris Bierkens and Sebastiano Grazzi and Kengo Kamatani and Gareth Roberts},
    year={2020},
    eprint={2006.13777},
    archivePrefix={arXiv},
}

@misc{bierkens2018highdimensional,
    title={High-dimensional scaling limits of piecewise deterministic sampling algorithms},
    author={Joris Bierkens and Kengo Kamatani and Gareth O. Roberts},
    year={2018},
    eprint={1807.11358},
    archivePrefix={arXiv},
}

@misc{mider2017continuousdiscrete,
    title={Continuous-discrete smoothing of diffusions},
    author={Marcin Mider and Moritz Schauer and Frank van der Meulen},
    year={2020},
    eprint={1712.03807},
    archivePrefix={arXiv},
}

@book{liptser2001statistics,
  title={Statistics of Random Processes: I. General Theory},
  author={Liptser, R.S. and Aries, B. and Shiryaev, A.N.},
  isbn={9783662130438},
  lccn={00041918},
  series={Stochastic Modelling and Applied Probability},
  year={2013},
  publisher={Springer Berlin Heidelberg}
}

@article{van2018adaptive,
  title={Adaptive nonparametric drift estimation for diffusion processes using {F}aber--{S}chauder expansions},
  author={van der Meulen, Frank and Schauer, Moritz and van Waaij, Jan},
  journal={Statistical Inference for Stochastic Processes},
  volume={21},
  number={3},
  pages={603--628},
  year={2018},
  publisher={Springer}
}

@article{bladt2014simple,
  title={Simple simulation of diffusion bridges with application to likelihood inference for diffusions},
  author={Bladt, Mogens and S{\o}rensen, Michael and others},
  journal={Bernoulli},
  volume={20},
  number={2},
  pages={645--675},
  year={2014},
  publisher={Bernoulli Society for Mathematical Statistics and Probability}
}

@book{grimmett2001probability,
title={Probability and random processes}, 
publisher={Oxford University Press},
author={Grimmett, Geoffrey and Stirzaker, David},
year={2001} 
}

@misc{andrieu2018hypocoercivity,
    title={Hypocoercivity of Piecewise Deterministic {M}arkov Process-{M}onte {C}arlo},
    author={Christophe Andrieu and Alain Durmus and Nikolas Nüsken and Julien Roussel},
    year={2018},
    eprint={1808.08592},
    archivePrefix={arXiv},
}

@article{vandermeulen2017,
author = "van der Meulen, Frank and Schauer, Moritz",
doi = "10.1214/17-EJS1290",
fjournal = "Electronic Journal of Statistics",
journal = "Electron. J. Statist.",
number = "1",
pages = "2358--2396",
publisher = "The Institute of Mathematical Statistics and the Bernoulli Society",
title = "Bayesian estimation of discretely observed multi-dimensional diffusion processes using guided proposals",
url = "https://doi.org/10.1214/17-EJS1290",
volume = "11",
year = "2017"
}

@article{diaconis2000analysis,
  title={Analysis of a nonreversible {M}arkov chain sampler},
  author={Diaconis, Persi and Holmes, Susan and Neal, Radford M},
  journal={Annals of Applied Probability},
  pages={726--752},
  year={2000},
  publisher={JSTOR}
}

@misc{andrieu2019peskuntierney,
    title={Peskun-{T}ierney ordering for {M}arkov chain and process {M}onte {C}arlo: beyond the reversible scenario},
    author={Christophe Andrieu and Samuel Livingstone},
    year={2019},
    eprint={1906.06197},
    archivePrefix={arXiv},
}

@book{mckean1969stochastic,
  title={Stochastic integrals},
  author={McKean, Henry P},
  volume={353},
  year={1969},
  publisher={American Mathematical Soc.}
}

@article{fearnhead2018,
author = "Fearnhead, Paul and Bierkens, Joris and Pollock, Murray and Roberts, Gareth O.",
doi = "10.1214/18-STS648",
fjournal = "Statistical Science",
journal = "Statist. Sci.",
month = "08",
number = "3",
pages = "386--412",
publisher = "The Institute of Mathematical Statistics",
title = "{Piecewise Deterministic {M}arkov Processes for Continuous-Time {M}onte {C}arlo}",
url = "https://doi.org/10.1214/18-STS648",
volume = "33",
year = "2018"
}

@article{beskos2006retrospective,
  title={Retrospective exact simulation of diffusion sample paths with applications},
  author={Beskos, Alexandros and Papaspiliopoulos, Omiros and Roberts, Gareth O and others},
  journal={Bernoulli},
  volume={12},
  number={6},
  pages={1077--1098},
  year={2006},
  publisher={Bernoulli Society for Mathematical Statistics and Probability}
}

@misc{mider2019simulating,
  title={Simulating bridges using confluent diffusions},
  author={Mider, Marcin and Jenkins, Paul A and Pollock, Murray and Roberts, Gareth O and S{\o}rensen, Michael},
  eprint={1903.10184},
  archivePrefix={arXiv},
  year={2019}
}

@article{bierkens2018simulation,
title={Simulation of elliptic and hypo-elliptic conditional diffusions},
volume={52},
DOI={10.1017/apr.2019.54},
number={1},
journal={Advances in Applied Probability},
publisher={Cambridge University Press},
author={Bierkens, Joris and van der Meulen, Frank and Schauer, Moritz},
year={2020},
pages={173--212}
}

@book{klebaner2005introduction,
  title={Introduction to stochastic calculus with applications},
  author={Klebaner, Fima C},
  year={2005},
  publisher={World Scientific Publishing Company}
}

@inproceedings{ge2018t,
  author    = {Hong Ge and
               Kai Xu and
               Zoubin Ghahramani},
  title     = {Turing: a language for flexible probabilistic inference},
  booktitle = {International Conference on Artificial Intelligence and Statistics,
               {AISTATS} 2018, 9-11 April 2018, Playa Blanca, Lanzarote, Canary Islands,
               Spain},
  pages     = {1682--1690},
  year      = {2018},
  url       = {http://proceedings.mlr.press/v84/ge18b.html},
}

@article{DUANE1987216,
title = "Hybrid Monte Carlo",
journal = "Physics Letters B",
volume = "195",
number = "2",
pages = "216 - 222",
year = "1987",
issn = "0370-2693",
doi = "10.1016/0370-2693(87)91197-X",
url = "http://www.sciencedirect.com/science/article/pii/037026938791197X",
author = "Simon Duane and A.D. Kennedy and Brian J. Pendleton and Duncan Roweth"
}

@article{karatzas1998brownian,
  title={Brownian motion and stochastic calculus},
  author={Karatzas, I and Shreve, Steven E},
  journal={Graduate texts in Mathematics},
  volume={113},
  year={1991}
}

@article{roberts2001inference,
  title={On inference for partially observed nonlinear diffusion models using the {M}etropolis--{H}astings algorithm},
  author={Roberts, Gareth O and Stramer, Osnat},
  journal={Biometrika},
  volume={88},
  number={3},
  pages={603--621},
  year={2001},
  publisher={Oxford University Press}
}

@article{bierkens2019,
author = "Bierkens, Joris and Fearnhead, Paul and Roberts, Gareth",
doi = "10.1214/18-AOS1715",
fjournal = "The Annals of Statistics",
journal = "Ann. Statist.",
month = "06",
number = "3",
pages = "1288--1320",
publisher = "The Institute of Mathematical Statistics",
title = "The {Z}ig-{Z}ag process and super-efficient sampling for {B}ayesian analysis of big data",
url = "https://doi.org/10.1214/18-AOS1715",
volume = "47",
year = "2019"
}

@book{davis1993markov,
  title={Markov Models \& Optimization},
  author={Davis, M.H.A.},
  isbn={9780412314100},
  lccn={lc92039557},
  series={Chapman \& Hall/CRC Monographs on Statistics \& Applied Probability},
  year={1993},
  publisher={Taylor \& Francis}
}

@ARTICLE{2015arXiv151002451B,
       author = {{Bouchard-C{\^o}t{\'e}}, Alexandre and {Vollmer}, Sebastian J. and
         {Doucet}, Arnaud},
        title = "{The Bouncy Particle Sampler: A Non-Reversible Rejection-Free Markov Chain {M}onte {C}arlo Method}",
         year = "2015",
        month = 10,
archivePrefix = {arXiv},
       eprint = {1510.02451},
       adsurl = {https://ui.adsabs.harvard.edu/\#abs/2015arXiv151002451B}
}
